\documentclass[english]{article}
\usepackage[T1]{fontenc}
\usepackage[latin9]{inputenc}
\usepackage{geometry}
\usepackage{mathtools}
\usepackage{amsmath}
\usepackage{amsthm}
\usepackage{amssymb}
\usepackage{graphicx}
\usepackage{caption}
\usepackage{setspace}
\usepackage{esint}
\usepackage{tikz}
\usepackage{tkz-graph}
\tikzstyle{vertex}=[circle, draw, fill=black, inner sep=0pt, minimum size=6pt]
\newcommand{\vertex}{\node[vertex]}
\makeatletter
\numberwithin{equation}{section}
\numberwithin{figure}{section}
\theoremstyle{plain}
\newtheorem{thm}{\protect\theoremname}[section]
  \theoremstyle{definition}
  \newtheorem{defn}[thm]{\protect\definitionname}
 \theoremstyle{definition}
 \newtheorem*{defn*}{\protect\definitionname}
  \theoremstyle{remark}
  \newtheorem{rem}[thm]{\protect\remarkname}
  \theoremstyle{plain}
  \newtheorem{lem}[thm]{\protect\lemmaname}
  \theoremstyle{plain}
  \newtheorem{cor}[thm]{\protect\corollaryname}
  \theoremstyle{plain}
  \newtheorem{prop}[thm]{\protect\propositionname}
  \theoremstyle{plain}
  \newtheorem{claim}[thm]{\protect\claimname}

\makeatother

\usepackage{babel}
  \providecommand{\claimname}{Claim}
  \providecommand{\corollaryname}{Corollary}
  \providecommand{\definitionname}{Definition}
  \providecommand{\lemmaname}{Lemma}
  \providecommand{\propositionname}{Proposition}
  \providecommand{\remarkname}{Remark}
\providecommand{\theoremname}{Theorem}

\author{Jonathan Breuer and Netanel Levi\footnote{Institute of Mathematics, The Hebrew University of Jerusalem, Jerusalem, 91904, Israel.
Supported in part by the Israel Science Foundation (Grant No.\ 399/16) and in part by the United States-Israel Binational Science Foundation
(Grant No.\ 2014337), Emails: jbreuer@math.huji.ac.il, netanel.levi@mail.huji.ac.il}}
\title{On the Decomposition of the Laplacian on Metric Graphs}

\begin{document}
\sloppy

\maketitle

\begin{abstract}
We study the Laplacian on family preserving metric graphs. These are graphs that have a certain symmetry that, as we show, allows for a
decomposition into a direct sum of one-dimensional operators whose properties are explicitly related to the structure of the graph. Such
decompositions have been extremely useful in the study of Schr\"odinger operators on metric trees. We show that the tree structure is not
essential, and moreover, obtain a direct and simple correspondence between such decompositions in the discrete and continuum case.
\end{abstract}

\section{Introduction}

The study of Schr\"odinger operators on graphs has drawn a considerable amount of attention in the past few decades. So much so, that any
attempt at a short comprehensive review is doomed to fail. We can only refer the reader to a few representative surveys and collections
\cite{BCFK, BK,EKKST, MW}. Besides arising naturally in many physical contexts, the setting of graphs offers a wide array of examples where a
variety of mathematical phenomena related to the effects of geometry on spectral properties may be studied. Of particular recent interest is
the setting of continuum (aka `metric') graphs. In this setting, a graph is seen as a one dimensional simplicial complex, where the line
segments have lengths, and are glued to each other at the relevant vertices. Functions are defined on the line segments, and the operator
studied is the one dimensional Laplacian on the line segments, with some prescribed boundary conditions at the vertices. This is the model
that will be at the focus of our attention here. We shall describe it formally in the next section.

While originating in chemistry and physics in the study of molecules and mesoscopic systems (such as waveguides, see e.g.\ \cite{BK} and
references therein), such continuum models (also known as quantum graphs) have drawn the attention of the spectral theory community and have
served as a platform for the study of various topics. These include trace formulas in quantum chaos \cite{GS}, isospectrality and its
association with geometry \cite{GuS}, Anderson localization and extended states \cite{ASW, ASW1, Br2, HP, Kl, Ta}, Hardy inequalities \cite{EFK1, NS},
eigenvalue estimates \cite{BL, BKKM, EFK} and others \cite{Ca, Ca1, Kuch}.

A useful method in the context of infinite metric trees (i.e.\ connected graphs with no cycles), that has been applied in several of the works
mentioned above, was introduced by Naimark and Solomyak in \cite{NS}. This method requires the tree to be spherically symmetric around a
particular vertex (the root) and involves a decomposition of the operator into a direct sum of one-dimensional operators whose structure is
directly related to the structure of the tree and to the boundary conditions at the vertices. A similar method exists in the discrete case,
where the graph is considered as a combinatorial object and the operator studied is the discrete Laplacian or the adjacency matrix (see, e.g.,
\cite{AF, Br, Br1, BF}). While the similarity between the decomposition in the continuous and discrete case is clear and lies in the
exploitation of the symmetry properties of the graph, it is important to note there are essential technical differences. Whereas the discrete
case involves
studying cyclic subspaces generated by specially chosen functions, the continuum case (as presented in \cite{NS, Sol}) involves defining the
relevant invariant subspaces directly and relies heavily on the structure of the tree.

It has recently been realized in \cite{BrKe} that in the discrete case, the tree structure is not essential for this decomposition. It is in
fact possible to carry out this procedure for a more general class of graphs that we call `family preserving' and whose definition we give below
(see Definition \ref{fpg})\footnote{\cite{BrKe} actually study a slightly more general class of graphs that they call `path commuting', but
as the definition is considerably more involved and less intuitive and since all relevant examples are family preserving we have decided to
prefer here simplicity to generality and restrict our attention to family preserving graphs.}. This class contains radially symmetric trees, antitrees (see Section 5 for the definition), and various other graphs. We refer the reader to \cite{BrKe} for examples and some graphics.

The objective of this work is to extend this decomposition for family preserving graphs to the
continuum case of metric graphs as well. Since, as remarked above, the standard decomposition
technique for metric trees relies heavily on the tree structure, one needs to take a different approach. A natural
approach to this task, and the one that we shall take, is to try and obtain a direct translation of the discrete decomposition to the decomposition
in the metric case. Such an approach has also the added bonus of making explicit the connection between the combinatorial and continuum decompositions, and as we shall show, can recover the original Naimark-Solomyak procedure from the procedure in the discrete case.

While employing the discrete decomposition scheme is indeed natural for getting a decomposition in the metric case, the results in this paper should by no means be viewed as a direct extension of those in \cite{BK}. First, as the functions in the metric case live on edges, whereas those in the discrete case live on vertices, it is a non-trivial task to adapt one procedure to the other. Second, as is evident in Section 4 (where we describe the decomposition algorithm) the discrete decomposition is only one of several steps towards obtaining the one-dimensional constituent operators in the direct sum, and a rather straightforward one at that. Finally, as the class of family preserving graphs is considerably larger than that of spherically symmetric rooted trees, we believe the results here open the door to constructing and studying a wide array of metric graphs with interesting and rich spectral properties.
\begin{rem}
	In the context of decomposing an operator using symmetries, we refer to the preprint \cite{BBJL}, which discusses quotients of finite-dimensional operators with respect to a representation of some symmetry group. Specifically, \cite[Section 6]{BBJL} discusses the case of compact quantum graphs, and the quotient operators are taken with respect to a representation of the underlying graph's symmetry group. It is reasonable to expect that the methods there extend to the infinite dimensional case studied here. A natural question then is whether one can use them in this case to reproduce our results. We note that it seems that our approach is somewhat more direct as it does not involve going through representations of the relevant symmetry group, (which is infinite here). On the other hand, the role of the symmetry is made more explicit in the approach of \cite{BBJL}. It would be interesting to study the relation between these two approaches further.
\end{rem}
The rest of this paper is structured as follows. After presenting some basic definitions, we describe our main abstract result (Theorem \ref{thm_main})
in the next section. As the proof of Theorem \ref{thm_main} relies on a connection between the discrete and continuous case, we devote the
first two subsections of Section 3 to some preliminary results on this connection, presenting the proof in Section 3.3. Although Theorem \ref{thm_main} is the main abstract result, the raison d'\^{e}tre of this paper lies in Section 4, where we describe the structure of the
components in the decomposition and their relation to the structure of the graph. Section 5 presents a demonstration that our approach reduces to the Naimark-Solomyak approach when $\Gamma$ is a tree and an application to the spectral analysis of spherically homogeneous (see Definition \ref{def:SH}) metric antitrees\footnote{The preprint \cite{KN} presents an extensive spectral analysis of spherically homogeneous metric antitrees, also using a decomposition inspired by \cite{BrKe}. Since spherically homogeneous metric antitrees are family preserving, our general setting includes such graphs as particular cases (although we restrict attention to the self-adjoint case). Our emphasis here, however, is on the decomposition method itself and not on the spectral analysis.}.

{\bf Acknowledgments} We thank Rami Band, Gregory Berkolaiko, Aleksey Kostenko and the anonymous referees for useful comments.


\section{Definitions and Statement of the Main Result}

\subsection{Some Definitions}

We begin with some definitions pertaining to the combinatorial structure of a graph. We shall add the continuous structure later.
A rooted graph is a graph $G=\left\langle V(G),E(G)\right\rangle$, with a special
vertex $o\in V(G)$, which is called the root. We assume that $V(G),E(G)$ (the vertex and edge set, respectively)
are infinite, and that for every $v\in V(G)$, $deg(v)<\infty$, where $\deg(v)$ is the degree of $v$(=the number of edges incident to $v$). We
also assume throughout that our graphs are simple (i.e.\ there are no multiple edges or loops) and connected.

For every $v\in V(G)$, we denote the set of edges incident to $v$ by $E_{v}$. A path of length $k$ between two vertices, $v,w \in V(G)$ is a
$k$-tuple of vertices $\left(v=v_1,v_2,v_3\ldots v_k=w \right)$ such that for each $j=1,\ldots, (k-1)$, $(v_j,v_{j+1}) \in E(G)$. The set of
all shortest paths between
two vertices $u,v\in V(G)$ will be denoted as $\langle u,v\rangle$. For $u,v,w\in V(G)$, we write $u\in\langle v,w\rangle$ if $u$ lies
on some shortest path between $v$ and $w$. Given a graph, the relation $u\preceq v\iff u\in\langle o,v\rangle$ induces a partial order on
the vertices, with $o$ being the minimal element. For $v\in V(G)$,
we define $gen(v)$ to be the length (i.e.\ the number of steps) of some shortest path between
$v$ and $o$ and $G_{v}$ to be the graph induced by the set of all vertices comparable with $v$ (with respect to $\preceq$). Namely, $G_v$ is the graph induced by the set of vertices, $w$, such that either $w\preceq v$ or $v\preceq w$. An illustration of some of these definitions is given in Figure \ref{ex_defs_disc}.

\begin{figure}
	\begin{center}
		\begin{tikzpicture}[scale=0.9]
		\vertex(o) at (3,1.5) {};
		\vertex(v) at (3,0) {};
		\draw (2.6,0) node{$v$};
		\vertex(w1) at (2,-1.5) {};
		\draw (1.6,-1.5) node{$w$};
		\vertex(w2) at (4,-1.5) {};
		\vertex(u1) at (1.5,-3) {};
		\draw (1.1,-3) node{$u_1$};
		\vertex(u2) at (2.5,-3) {};
		\draw (2.9,-3) node{$u_2$};
		\vertex(u3) at (3.5,-3) {};
		\vertex(u4) at (4.5,-3) {};
		\vertex(x1) at (2,-4.5) {};
		\vertex(x2) at (4,-4.5) {};
		\vertex(x3) at (2,-5.5) {};
		\vertex(x4) at (4,-5.5) {};
		\node[vertex, draw=none, fill=white](dots1) at (2, -6.5) {};
		\node[vertex, draw=none, fill=white](dots2) at (4, -6.5) {};
		\path (x3) -- (dots1) node [black, font=\Huge, midway, sloped] {$\dots$};
		\path (x4) -- (dots2) node [black, font=\Huge, midway, sloped] {$\dots$};
		\draw (3,1.5) -- (3,0)[dashed];
		\draw (3,0) -- (2, -1.5)[dashed];
		\Edge (v)(w2);
		\draw (2,-1.5) -- (1.5,-3)[dashed];
		\Edge (w1)(u2);
		\Edge (w2)(u3);
		\Edge (w2)(u4);
		\draw (1.5,-3) -- (2,-4.5)[dashed];
		\Edge (u2)(x1);
		\Edge (u3)(x2);
		\Edge(u4)(x2);
		\draw (2,-4.5) -- (2,-5.5)[dashed];
		\Edge (x2)(x4);
		\end{tikzpicture}
	\end{center}
	\captionof{figure}{In the graph above, the dashed edges represent $G_{u_1}$, $E_w=\left\{\left(v,w\right),\left(w,u_1\right),\left(w,u_2\right)\right\},$ $gen(v)=1$, $gen(w)=2$ and $gen(u_1)=gen(u_2)=3$.}
	\label{ex_defs_disc}
\end{figure}
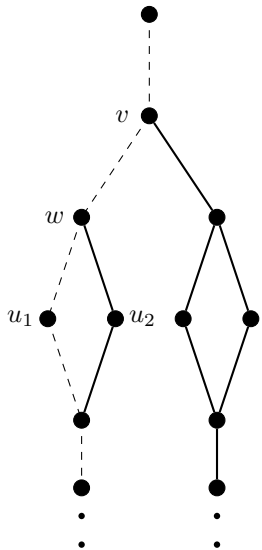

Define $S_{n}=\left\{v\in V(G)\,:\,gen(v)=n\right\}$
(namely, $S_{n}$ is the set of vertices with distance $n$ from the
root). We assume throughout that for any $n$, there are no edges between vertices in $S_n$. With this assumption and the partial order
structure induced on a graph, an edge
$(u,v)=e\in E(G)$ has an initial vertex, namely the vertex closer to $o$, which is denoted by $i(e)$, and a terminal
vertex, denoted by $t(e)$. We also define $gen(e)=gen(i(e))$.

The discrete Laplacian on a graph $G$ is the densely defined operator $\Delta_{d}:\,D(\Delta_{d})\rightarrow\ell^{2}(G)$,
defined by $\Delta_{d}\phi(x)=\underset{x\sim y}{\sum}\left(\phi(x)-\phi(y)\right)$,
where $x\sim y$ means that $x$ and $y$ are neighbors, and \mbox{$D(\Delta_{d})=\left\{\phi\in\ell^{2}(G)\,:\,\Delta\phi\in\ell^{2}(G)\right\}$}.
$D(\Delta_{d})$ is dense because it contains all of the functions
with compact (finite) support.

In order to discuss some symmetry properties of a graph, one would
want to be able to say when any two vertices of the same sphere (i.e.\
two vertices in $S_{n}$) ``look alike''. The definition of a \emph{rooted
graph automorphism} serves that purpose.
\begin{defn}
A \emph{rooted graph automorphism} is a bijection $\tau:\,V(G)\rightarrow V(G)$
such that $\tau(o)=o$, and $\left(\tau(v),\tau(u)\right)\in E\iff\left(v,u\right)\in E$.
\end{defn}

A rooted graph $G=\left\langle V(G),E(G)\right\rangle$ will be called \emph{spherically
symmetric} if for every $n\in\mathbb{N}$ and for every $v,u\in S_{n}$,
there exists a rooted graph automorphism $\tau$, such that $\tau(v)=u$.
In other words, the group of all rooted graph automorphisms on $G$
acts transitively on $S_{n}$.

Two vertices $u,v\in S_{n}$ will be called \emph{forward neighbors} if there
exists some $w\in S_{n+1}$ such that $(v,w),(u,w)\in E$. Similarly,
$u,v\in S_{n}$ will be called \emph{backward neighbors} if there exists
some $w\in S_{n-1}$ such that $\left(v,w\right),\left(u,w\right)\in E$. We can now define the type of symmetry we need.
\begin{defn} \label{fpg}A rooted graph $G$ will be called \emph{family preserving} if
the following conditions hold:

$(i)$ If $u,v\in S_{n}$ are backward neighbors, then there exists
a rooted graph automorphism $\tau$, such that $\tau(u)=v$, and $\tau|_{S_{n-j}}=Id$
for all $n \geq j\geq1$.

$(ii)$ if $u,v\in S_{n}$ are forward neighbors, then there exists
a rooted graph automorphism $\tau$ such that $\tau(u)=v$, and $\tau|_{S_{n+j}}=Id$
for all $j\geq1$.
\end{defn}

\begin{rem}
Family preserving graphs were introduced in \cite{BrKe} in the discrete context as graphs on which it is possible to obtain a constructive
decomposition of the combinatorial Laplacian into operators over cyclic subspaces. Examples are given in Figure \ref{ex_defs_disc} above and in Section 4 below, where we describe the decomposition procedure. We shall also present more examples in Section 5 below. We
refer the reader to \cite{BrKe} for more illustrations and further discussion. We remark here only that in the case that $G$ is a tree then Definition
\ref{fpg} is equivalent to spherical symmetry, whereas for general graphs, this is a strictly stronger property. In Figure \ref{ex_sph_sym} we give an example of a graph which is spherically symmetric but not family preserving.
\end{rem}

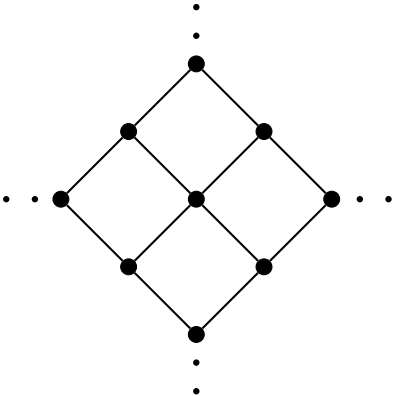
\begin{figure}
	\begin{center}
		\begin{tikzpicture}[scale=0.9]
		\vertex(o) at (0,0) {};
		\vertex(v1) at (1,1) {};
		\vertex (v2) at (-1,1) {};
		\vertex (v3) at (-1,-1) {};
		\vertex (v4) at (1,-1) {};
		\vertex (w1) at (2,0) {};
		\vertex (w2) at (0,2) {};
		\vertex (w3) at (-2,0) {};
		\vertex (w4) at (0,-2) {};
		\node[vertex, draw=none, fill=white](dots1) at (3, 0) {};
		\node[vertex, draw=none, fill=white](dots2) at (0, 3) {};
		\node[vertex, draw=none, fill=white](dots3) at (-2.6,0) {};
		\node[vertex, draw=none, fill=white](dots4) at (0,-3) {};
		\path (w1) -- (dots1) node [black, font=\Huge, midway, sloped] {$\dots$};
		\path (w2) -- (dots2) node [black, font=\Huge, midway, sloped] {$\dots$};
		\path (w3) -- (dots3) node [black, font=\Huge, midway, sloped] {$\dots$};
		\path (w4) -- (dots4) node [black, font=\Huge, midway, sloped] {$\dots$};
		\Edge (o)(v1);
		\Edge (o)(v2);
		\Edge (o)(v3);
		\Edge (o)(v4);
		\Edge (v1)(w1);
		\Edge (v1)(w2);
		\Edge (v2)(w2);
		\Edge (v2)(w3);
		\Edge (v3)(w3);
		\Edge (v3)(w4);
		\Edge (v4)(w4);
		\Edge (v4)(w1);
		\end{tikzpicture}
	\end{center}
	\captionof{figure}{An example of a spherically symmetric graph which is not family preserving. The middle vertex is the root. The dots represent some spherically symmetric continuation (e.g.\ by line segments).}
	\label{ex_sph_sym}
\end{figure}

Having presented the basic definitions we need for the underlying combinatorial structure, we can now add a continuous structure on the edges.
Consider a rooted graph $G$ and identify each edge, $e\in E(G)$, with a non degenerate line segment$\subseteq \mathbb{R}$. That edge is now a
metric space (with the usual metric on $\mathbb{R}$), and in particular its length, denoted $l(e)$, is that of the associated line segment. We assume\footnote{This assumption is made in order to avoid technical issues regarding the definition of the domain of the operator and we make it here for simplicity. The central ideas in our approach do not rely on this assumption.} throughout that $\underset{e\in{E(G)}}\inf l(e)>0$.
This metric structure on the edges together with the gluing at the vertices defines a length for any path on the resulting topological space. This allows one to define a metric $d$ on the graph, under which the distance between two points is the length of a minimal path between them. To avoid confusion with the underlying combinatorial structure, we denote this metric space $\Gamma$ and refer to such a structure as a
\emph{metric graph}. For $x\in\Gamma$ we denote $|x|=d(x,o)$ (note that $x$ here is not necessarily a vertex of $G$, the underlying combinatorial
graph). In the sequel, when we want to refer to the combinatorial structure of a metric graph, $\Gamma$, we shall use the notation $G_\Gamma$
(though we may omit the subscript when there is no risk of confusion). For simplicity, we let $V(\Gamma)= V\left(G_\Gamma \right)$ and
$E(\Gamma)=E\left(G_\Gamma \right)$.

\begin{defn} \label{def:SH}
A rooted metric graph $\Gamma$ will be called spherically homogeneous
if for every $e_{1},e_{2}\in E(\Gamma)$, $gen(e_{1})=gen(e_{2})\Rightarrow l(e_{1})=l(e_{2})$.
\end{defn}

In this work, unless stated otherwise, a metric graph $\Gamma$ will always be spherically homogeneous. For any such graph, the length of an
edge and the distance of a vertex from the root depend solely on their generation. Thus, we denote by $t_k$ the distance of a vertex in $S_k$
from the root. Also, letting $l_k=l(e)$ for some $e$ with $gen(e)=k$ and following \cite{NS,Sol}, we may define the height of $\Gamma$ by
\begin{equation} \nonumber
h(\Gamma)=\sum_{k=0}^\infty l_k
\end{equation}
which may be finite or infinite.

\begin{defn} \label{def:FPM}
A rooted metric graph $\Gamma$ will be called family preserving if it is spherically homogeneous and if $G_\Gamma$ is family preserving.
\end{defn}

The underlying metric structure induces a measurable structure on $\Gamma$. We denote the relevant $\sigma$-algebra by $\mathcal{B}$. The
measure of a set $E\in\mathcal{B}$ is $\underset{e\in E(\Gamma)}{\sum}\lambda(e\cap E)$,
where $\lambda$ is the Lebesgue measure on $e$. We will denote that
measure by $\mu$. $L^{2}(\Gamma)$ is the space of all measurable
functions (with respect to $\mathcal{B}$) $f:\,\Gamma\rightarrow\mathbb{C}$
such that $\int_{\Gamma}|f|^{2}d\mu<\infty$.

Given a metric graph $\Gamma$, we denote by $\mathcal{H}$ the space
$L^{2}(\Gamma)$. Recall that given a domain $\Omega\subseteq\mathbb{R}^{d}$
and $k\in\mathbb{N}$, $H^{k}(\Omega)$ is the space of all functions
$f\in L^{2}(\Omega)$ for which all of the weak derivatives up to
order $k$ are in $L^{2}(\Omega)$.

We define a quadratic form $q:\,Q\times Q\rightarrow\mathbb{C}$ by
\[
Q=\left\{f\in L^{2}(\Gamma)\text{ s.t. \ensuremath{f} is continuous, \ensuremath{f|_{e}\in H^{1}(e)\text{ for every \ensuremath{e\in
E(\Gamma)\text{ \ensuremath{f(o)=0\text{, and \ensuremath{\int_{\Gamma}|f'(x)|^{2}d\mu(x)<\infty}}}}}}}}\right\}
\]
and for $f,g\in Q$, $q(f,g)=\intop_{\Gamma}f'\overline{g'}d\mu$.
The (Dirichlet-Kirchhoff) Laplacian, $\Delta_{c}^{\Gamma}$, on $\Gamma$ is the self adjoint
operator associated with this quadratic form. By \cite[Theorem 1.4.19]{BK}, and under the assumption $\underset{e\in{E(\Gamma)}}\inf l(e)>0$, the domain of $\Delta_{c}^{\Gamma}$
is the space of all functions $f\in L^2(\Gamma)$ such that $f(o)=0$, $f$ is continuous,
$f|_{e}\in H^{2}(e)$ for every $e\in E(\Gamma)$, $\int_{\Gamma}|f''(x)|^{2}dx<\infty$,
and for every $v\in V(\Gamma)$, $\underset{u\sim v}{\sum}\left(f|_{[u,v]}\right)'(v)=0$,
where $[u,v]$ is the line segment connecting $u$ and $v$. These
conditions are called the Kirchhoff boundary conditions \cite{NS}.
When it is clear from the context, we will omit the subscript and
superscript and just write $\Delta$.
\begin{rem}
	Although in this work we only consider the Kirchhoff boundary conditions on the vertices, there are other possible vertex conditions which can be defined in order to make the Laplacian on $\Gamma$ self adjoint (see e.g.\ \cite[Chapter 1]{BK}). As long as these vertex conditions are chosen so as to preserve the underlying symmetry of the operator (i.e., in a spherically symmetric fashion), it is clear that our results, properly modified, extend to that case as well.
\end{rem}

\subsection{The Main Abstract Result}

Our main result is a decomposition of the Laplacian into invariant subspaces described by the structure of the graph. These spaces are
generated by functions arising in the decomposition of the corresponding discrete structure. The discrete decomposition provides
the skeleton for the decomposition in the metric case. Thus, before describing our main result, we need to recall the result in the discrete
case.

Let $G$ be a family preserving graph. Recall that $G_{v}$ set of all vertices comparable with $v$ with respect to the order relation $\preceq$
introduced in the previous subsection. Given $n,j\in\mathbb{N}$, introduce the following operators:

\begin{equation}
E_{n}:\,\ell^{2}(S_{n})\rightarrow\ell^{2}(S_{n+1}),\,E_{n}(\varphi)(v)=\underset{x\in S_{n}\cap G_{v}}{\sum}\varphi(x)=\underset{x \in S_n,
x\sim v} \sum\varphi(x)
\end{equation}

\begin{equation}
\Lambda_{n,+j}:\,\ell^{2}(S_{n})\rightarrow\ell^{2}(S_{n}),\,\Lambda_{n,+j}=E_{n}^{T}\cdots E_{n+j}^{T}\cdot E_{n+j}\cdots E_{n}
\end{equation}

and if $j\leq{n}$, then we may also define

\begin{equation}
\Lambda_{n,-j}:\,\ell^{2}(S_{n})\rightarrow\ell^{2}(S_{n}),\,\Lambda_{n,-j}=E_{n-1}\cdots E_{n-j}\cdot E_{n-j}^{T}\cdots E_{n-1}^{T}
\end{equation}
A decomposition for the Laplacian on family preserving graphs was presented in \cite{BrKe}.
\begin{thm}
\emph{(\cite[Theorem 2.6]{BrKe})} \label{thm_br} Let $G$ be a
family preserving graph. Then $\ell^{2}(G)=\oplus_{r=0}^{\infty}H_{r}$ where $H_r$ is invariant under $\Delta_d$ and
such that:

$(i)$ For every $r$ there exists $n(r)$ and a vector $\phi_{0}^{r}$
such that $supp(\phi_{0}^{r})\subseteq S_{n(r)}$, and
$H_{r}=\overline{span\left\{\phi_{0}^{r},\Delta_{d}\phi_{0}^{r},\Delta_{d}^2\phi_{0}^{r},\ldots\right\}}$.

$(ii)$ The set $\{\phi_{0}^{r},\phi_{1}^{r},\ldots\}$ obtained from
$\{\phi_{0}^{r},\Delta_{d}\phi_{0}^{r},\Delta_{d}^{2}\phi_{0}^{r},\ldots\}$
by applying the Gram-Schmidt process has the property that $supp(\phi_{k}^{r})\subseteq S_{n(r)+k}$
for every $k\geq0$.

Furthermore, for every $r,j\in\mathbb{N}$, $\phi_{0}^{r}$ is an
eigenfunction of $\Lambda_{n(r),\pm{j}}$.
\end{thm}

\begin{rem}
Throughout this work, the notation $\phi_{n}^{r}$ will refer to the
functions presented in Theorem \ref{thm_br}, $(ii)$.
\end{rem}

\begin{rem}
The last statement, about $\phi_0^r$ being an eigenfunction of $\Lambda_{n(r),\pm{j}}$, is not part of the statement of Theorem 2.6 in
\cite{BrKe}.
However, it is shown in its proof. We include it here in the statement for simplicity of reference.
\end{rem}

\begin{rem}
\label{remark_Br} It also follows from the proof of \cite[Theorem 2.6]{BrKe} that for every $k\in\mathbb{N}$, there exists $c\in\mathbb{C}$
such that $\left(\Delta_{d}^{k}(\phi_{0}^{r})\right)|_{S_{n+k}}=c\phi_{k}^{r}$.
\end{rem}

Now let $\Gamma$ be a family preserving metric graph and $G=G_\Gamma$ the associated discrete graph. For $x \in \Gamma$ define $G_{x}$ to be
$G_{i(e)}\cap{G_{t(e)}}$ where $x$ lies on the edge $e$ (recall that $i(e)$ is the initial vertex of $e$, and $t(e)$ is the terminal vertex of $e$).
Note that for any $v \in V(G)$ and $0<t<h(\Gamma)$ the set $G_{v}\cap\left\{x\in\Gamma\,:\,|x|=t\right\}$ is finite. Moreover, if $|u|=|v|$ then, by
symmetry, the size of this set is equal to the size of $G_{u}\cap\left\{x\in\Gamma\,:\,|x|=t\right\}$. For $|v|=n$ we denote this number by $g_{n}(t)$.
Finally, we let $g_\Gamma(t)=g_0(t)= \# \left\{x \in \Gamma \mid |x|=t \right\}$.

The first step in translating the discrete decomposition into a continuous one is defining a procedure to obtain a function on the metric
graph from one defined on the discrete graph. We do this by `spreading' the values of the function over the graph, taking into account the
symmetry. To be precise, given $n\in\mathbb{N}$ and $\phi\in\ell^{2}(S_{n})$, let $h^{\phi}:\,\Gamma\rightarrow\mathbb{C}$
be defined by
\begin{equation} \label{eq:hphi}
h^{\phi}(x)=\frac{1}{\sqrt{g_{n}(|x|)\cdot|G_{x}\cap S_{n}|}}\underset{v\in G_{x}\cap S_{n}}{\sum}\phi(v)
\end{equation}

Next we want to use this procedure to define projection operators based on the functions $\phi_n^r$ above. For any
$f:\,\Gamma\rightarrow\mathbb{C}$, let $f_{t}\in\mathbb{C}^{g_{\Gamma}(t)}$
be the vector whose entries are the values $f$ takes on points with distance $t$
from the root (with respect to some fixed ordering of these points). For $\phi \in \ell^2(S_n)$, let $P_{\phi}:\,L^{2}(\Gamma)\rightarrow
L^{2}(\Gamma)$ be defined
by
\begin{equation}
P_{\phi}(f)(x)=\left \langle f_{|x|},h_{|x|}^{\phi} \right \rangle h^{\phi}(x)
\end{equation}
where $\langle\cdot,\cdot\rangle$ is the standard inner product in
$\mathbb{C}^{g_{\Gamma}(|x|)}$. We shall show in the next section that
for every $r,n\in\mathbb{N}$, $P_{\phi_{n}^{r}}$ is in fact an
orthogonal projection (with image in $L^2(\Gamma)$). We define the space $F_{r,n}$
to be the image of the orthogonal projection $P_{r,n}\coloneqq P_{\phi_{n}^{r}}$.
The decomposition presented in this work will include the image of
$P_{r,0}$ for every $r\in\mathbb{N}$. To abbreviate, we shall denote
$h^{\phi_{n}^{r}}=h^{r,n}$, $h^{r,0}=h^{r}$, $P_{r,0}=P_{r}$, and $F_{r,0}=F_r$.

We want to use the spaces $F_r$ to decompose $\Delta$ on $\Gamma$. There is, however, a `local dimension counting' issue: while functions on
$\Gamma$ are determined by the values they take on edges, functions on $G_\Gamma$ are determined by their values on vertices. If $G$ is a
tree, this is not an issue since the number of edges of a particular generation is always equal to the number of vertices of the next
generation. Generally, however, the local number of vertices and edges does not need to be the same (think of antitrees described in Section
5.2). In order to deal with this problem, in addition to the assumption that there are no edges between edges of the same generation, we need the following

\begin{defn}
A metric graph $\Gamma$ will be called \emph{locally balanced} if for every $n\in\mathbb{N}$,
\[
\begin{array}{c}
\#\left\{e\in E\,:\,gen(e)=n\right\}=\#\left\{v\in V\,:\,gen(v)=n\right\}\\
\text{or}\\
\#\left\{e\in E\,:\,gen(e)=n\right\}=\#\left\{v\in V\,:\,gen(v)=n+1\right\}
\end{array}
\]
\end{defn}

It is intuitively clear that any graph can be made into a locally balanced graph by adding vertices in the middle of `bad' edges. Since such
vertices come with Kirchhoff boundary conditions, this makes no difference as far as $\Delta$ is concerned. Thus

\begin{prop}
\label{u_eq_loc_bal}Let $\Gamma$ be a metric family preserving graph,
then there exists another family preserving graph $\overline{\Gamma}$
such that the pair $\left\langle\Gamma,\Delta_{\Gamma}\right\rangle,\,\left\langle\overline{\Gamma},\Delta_{\overline{\Gamma}}\right\rangle$
are unitarily equivalent, and $\overline{\Gamma}$ is locally balanced.
\end{prop}

A proof is given in the appendix. Now, let $\Gamma$ be a metric family preserving graph and let $\overline{\Gamma}$ be a corresponding locally
balanced family preserving graph. Note that although $\Delta_\Gamma$ is unitarily equivalent to $\Delta_{\overline{\Gamma}}$, the spaces $F_r$
for each one of these graphs may be different.

We are finally ready to state our main result.
\begin{thm} \label{thm_main}
Let $\Gamma$ be a locally balanced family preserving metric graph satisfying $\underset{e\in{E(\Gamma)}} \inf l(e)>0$. Assume further that there are no edges between vertices of the same
generation in $\Gamma$.
Then the subspaces
$(F_{r})_{r\in\mathbb{N}}$ form a decomposition of $L^{2}(\Gamma)$,
which reduces the Laplacian. Meaning:

$(i)$ $L^{2}(\Gamma)=\underset{r}{\oplus}F_{r}$.

$(ii)$ For every $r$, the projection $P_{r}$ onto $F_{r}$ commutes
with the Laplacian, i.e.\ $P_{r}(D(\Delta))\subseteq D(\Delta)$,
and $P_{r}\Delta=\Delta P_{r}$ on $D(\Delta)$.

Furthermore, for every $r\in{\mathbb{N}}$, there exist $a_r\in{\left(0,\infty\right)}$ and $b_r\in{\left(0,\infty\right]}$, which depend on $\phi_{0}^{r}$, such that functions in $F_r$ are supported on $\left\{x\in{\Gamma} : a_r<|x|<b_r\right\}$ and $\Delta|_{F_r}$ is unitarily equivalent to a differential operator on $L^2\left(a_r,b_r\right)$.
\end{thm}
\begin{rem} \label{rem_frs}
	The actual strength of Theorem \ref{thm_main} is the simple description of the spaces $\left(F_r\right)_{r\in{\mathbb{N}}}$. By this we mean that the resulting operator on $L^2\left(a_r,b_r\right)$ is the Laplacian with boundary conditions along some set of points. The boundary onditions and the set on which they are defined can be relatively easily described via structural properties of the underlying graph. This, along with the proof of the second part of Theorem \ref{thm_main}, is shown in Section 4.
\end{rem}

\section{Proof of Theorem \ref{thm_main}}

Before proving Theorem \ref{thm_main}, we need some results regarding
the discrete structure of a family preserving graph, and some properties
of the functions presented in Theorem \ref{thm_br}.

\subsection{The Discrete Structure of a Family Preserving Graph}

In order to express the type of symmetry family preserving graphs
have, we will first need some definitions.
\begin{defn}
Let $n,k\in\mathbb{N}$ and let $v,u\in S_{n}$. A \emph{$k$-forward
path} from $v$ to $u$ is a path $(x_{0},\ldots,x_{k},\ldots,x_{2k})$
of length $2k+1$ such that $x_{0}=v$, $x_{2k}=u$, and $gen(x_{j})=n+j$
for all $j\leq k$.

Similarly, a \emph{$k$-backward path} from $v$ to $u$ is a path $(x_{0},\ldots,x_{k},\ldots,x_{2k})$
of length $2k+1$ such that $x_{0}=v$, $x_{2k}=u$, and $gen(x_{j})=n-j$
for all $j\leq k$.
\end{defn}

A $k$-forward path between $x,y\in S_{n}$ is a path that starts
at $x$, takes $k$ steps forward (i.e.\ away from the root), and
then takes $k$ steps backwards and reaches $y$. A $k$-backward path starts at $x$, takes $k$ steps towards the root and then takes $k$ steps
back to $y$.
\begin{defn}
Let $n,k\in\mathbb{N}$ and let $u\in S_{n},\,v\in S_{n+k}$. A \emph{descending
path} between $u$ and $v$ is a path $(x_{0}=u,\ldots,x_{k}=v)$,
where $gen(x_{i})=gen(x_{i-1})+1$ for every $1\leq i\leq k$. An
\emph{ascending path} is a path $(y_{0}=v,\ldots,y_{k}=u)$ where
$gen(y_{i})=gen(y_{i-1})-1$ for every $1\leq i\leq k$.
\end{defn}

We will use the following lemma, proved in \cite{BrKe}.
\begin{lem}(\cite[Lemma 3.4]{BrKe}) \label{k_for_path}Let $G$ be a family preserving graph. Let $x,y\in S_{n}$,
and assume there is a $k$-forward path between $x$ and $y$. Then
there is a rooted graph automorphism $\tau$, such that $\tau(x)=y$,
and $\tau|_{S_{n+k}}=Id$.

Similarly, assume there is a $k$-backward path between $x$ and $y$. Then
there is a rooted graph automorphism $\tau$, such that $\tau(x)=y$,
and $\tau|_{S_{n-k}}=Id$
\end{lem}

\begin{cor}
\label{same_num_paths}Let $G$ be a family preserving graph. Let
$x,y\in S_{n}$, and assume that there exists a $k$-forward path
between $x$ and $y$. Assume also that $x_{k}=v\in S_{n+k}$. Then
the number of descending paths between $x$ and $v$ is equal to the
number of descending paths between y and v.

Similarly, assume there exists a $k$-backward path between $x$ and $y$ and assume that $x_k=v \in S_{n-k}$. Then the number of ascending paths
between $x$ and $v$ is equal to the number of ascending paths between $y$ and $v$.
\end{cor}

\begin{proof}
We will prove the first part. The second one is analogous. Let $\tau$ be a rooted graph automorphism such that $\tau(x)=y$,
and $\tau|_{S_{n+k}}=Id$. Then $\tau$ takes every descending path
between $x$ and $v$ to a descending path between $y$ and $v$.
The fact that $\tau$ is bijective implies that this correspondence
is also bijective.
\end{proof}
\begin{lem}
\label{disj_or_equal}Let $\Gamma$ be a family preserving graph,
and let $n\in\mathbb{N}$. For every \mbox{$k\in \{-n,-n+1,\ldots,0,1,2,\ldots \}$} and $v,u\in S_{n+k}$,
either
\begin{equation} \nonumber
(i)\,(G_{v}\cap S_{n})\bigcap(G_{u}\cap S_{n})=\emptyset
\end{equation}
or
\begin{equation} \nonumber
(ii)\,G_{v}\cap S_{n}=G_{u}\cap S_{n}.
\end{equation}
\end{lem}

\begin{proof}
Let $v,u\in S_{n+k}$. Suppose $(G_{v}\cap S_{n})\bigcap(G_{u}\cap S_{n})\neq\emptyset$,
and let $w\in(G_{v}\cap S_{n})\bigcap(G_{u}\cap S_{n})$. Let $x\in G_{v}\cap S_{n}$.
If $k>0$ then, by the assumption, there is a $k$-forward path between $x$
and $w$. Thus by Lemma \ref{k_for_path}, there is a rooted graph
automorphism $\tau$ such that $\tau(w)=x$, and $\tau|_{S_{n+k}}=Id$.
There is a descending path $(w=x_0,x_{1},\ldots,x_{k}=u)$, which
translates under $\tau$ to a descending path $(x=\tau(x_{0}),\tau(x_{1}),\ldots,\tau(x_{k})=u)$,
which implies that $x\in G_{u}\cap S_{n}$. Since $x$ was chosen
arbitrarily from $G_{v}\cap S_{n}$, we have that $G_{v}\cap S_{n}\subseteq G_{u}\cap S_{n}$.
It can be similarly proven that $G_{u}\cap S_{n}\subseteq G_{v}\cap S_{n}$,
and we have $G_{v}\cap S_{n}=G_{u}\cap S_{n}$, as required.

If $k \leq 0$ the proof is similar.
\end{proof}
Lemma \ref{disj_or_equal} is the first step in obtaining the orthogonality of the spaces $(F_{r})_{r\in\mathbb{N}}$ as a consequence of the
orthogonality of the functions in the set $\{\phi_{n}^{r}\,:\,r,n\in\mathbb{N}\}$.

\begin{lem} \label{En}
Let $G$ be a family preserving graph, $j,n\in\mathbb{N}$, and $\phi\in\ell^{2}(S_{n})$.
Then $\forall u\in S_{n+j+1}$,
\[
(E_{n+j}\cdots E_{n})(\phi)(u)=\underset{x\in S_{n}\cap G_{u}}{\sum}k_{u}(x)\phi(x)
\]
 where $k_{u}(x)$ is the number of different descending paths from
$x$ to $u$.
\end{lem}

\begin{proof}
By induction on $j$. For $j=1$, the claim is trivial. Assume the claim is true for $j=k$. Then for $j=k+1$,
\begin{equation} \nonumber
\begin{split}
(E_{n+k+1}\cdots E_{n})(\phi)(u)=E_{n+k+1}(E_{n+k}\cdots E_{n})(\phi)(u)&=\underset{y\in S_{n+k}\cap G_{u}}{\sum}\underset{x\in S_{n}\cap
G_{y}}{\sum}k_{y}(x)\phi(x) \\
&=\underset{x\in S_{n}\cap G_{u}}{\sum}k_{u}(x)\phi(x).
\end{split}
\end{equation}
\end{proof}
\begin{lem}
\label{EnT}Let $G$ be a family preserving graph, $j,n\in\mathbb{N}$, and $\phi\in\ell^{2}(S_{n+j})$.
Then $\forall v\in S_{n}$,
\[
\left(E_{n}^{T}\cdots E_{n+j}^{T}\right)(\phi)(v)=\underset{x\in S_{n+j}\cap G_{v}}{\sum}k^{v}(x)\phi(x)
\]
where $k^{v}(x)$ is the number of different ascending paths from
$x$ to $v$.
\end{lem}

The proof is symmetric to the proof of Lemma \ref{En}.

\begin{rem}
\label{remark_En}It is easy to see that by the definition of $\Delta_{d}$,
given $\phi\in\ell^{2}(S_{n})$, $k\in\mathbb{N}$, for $v\in S_{n+k}$
we have $\Delta_{d}^{k}(\phi)(v)=(E_{n+k-1}\cdots E_{n})(\phi)(v)$.
\end{rem}

The fact that for every $r,k\in\mathbb{N}$, $\phi_{0}^{r}$ is an
eigenvector of $\Lambda_{n(r),+k}$ means that $F_{r}=F_{r,n}$ for
every $n\in\mathbb{N}$. This is proved in the following lemmas, and
is necessary in order to prove that for every family preserving metric
graph $\Gamma$, the spaces $(F_{r})_{r\in\mathbb{N}}$ span $L^{2}(\Gamma)$.
The main idea is that on family preserving graphs, for every $\phi\in\ell^{2}(S_{n})$,
$\Lambda_{n,\pm{k}}(\phi)$ takes the same values on vertices that share
a $k$-forward/backward path. This is shown in the following lemma.

\begin{lem}
\label{lambda_const}Let $G$ be a family preserving graph and let
$v,u\in S_{n}$ such that there exists a $k$-forward path between
$v$ and $u$. Then for every $\phi\in\ell^{2}(S_{n})$, $\Lambda_{n,+k}(\phi)(v)=\Lambda_{n,+k}(\phi)(u)$.
Furthermore, denote by $K$ the set of vertices in $S_{n}$ that share
a $k$-forward path with $v$ (or u). Then $\Lambda_{n,+k}(\phi)|_{K}\equiv c\underset{x\in K}{\sum}\phi(x)$,
where $c$ is some natural number greater than 1, which depends only on $n,k$. The analogous statement for $\Lambda_{n,-k}$ and for
$v,u\in{S_n}$ that have a $k$-backward path between them is also true.
\end{lem}

\begin{proof}
	We prove the lemma for $v,u\in{S_n}$ that share a $k$-forward path. The proof for the case that they share a $k$-backward path is
symmetric.
	
By Lemma \ref{disj_or_equal} we have
\[
G_{v}\cap S_{n+k}=G_{u}\cap S_{n+k}\coloneqq A_{v,u}.
\]
In addition, by the same lemma, for every $w_{1},w_{2}\in A_{v,u}$,
\[
G_{w_{1}}\cap S_{n}=G_{w_{2}}\cap S_{n}\coloneqq B_{v,u}.
\]
Denote $\widetilde{\phi}=E_{n+k}\cdots E_{n}(\text{\ensuremath{\phi})\ensuremath{\in\ell}}^{2}(S_{n+k})$.
By Lemma \ref{En}, for every $w\in A_{v,u}$,
\[
\widetilde{\phi}(w)=\underset{x\in B_{v,u}}{\sum}k_{w}(x)\phi(x).
\]
Now, by Corollary \ref{same_num_paths}, the factor $k_{w}(x)$ in
the sum is constant, so we may write
\begin{equation}
\widetilde{\phi}(w)=c_{1}\underset{x\in B_{v,u}}{\sum}\phi(x)
\end{equation}
where $c_{1}$ is $k_{w}(x)$ for some $x\in B_{v,u}$. Now, again
by Corollary \ref{same_num_paths}, the constant $c_{1}$ is the same
for every $w_{1},w_{2}\in A_{v,u}$, and that means that $\widetilde{\phi}$
is constant on $A_{v,u}$. By definition, we have
\[
\Lambda_{n,+k}(\phi)(v)=E_{n}^{T}\cdots E_{n+k}^{T}\left( \widetilde{\phi} \right)(v)=\underset{x\in
A_{v,u}}{\sum}k^{v}(x)\widetilde{\phi}(x).
\]
We apply Corollary \ref{same_num_paths} twice again, and get that $k^{v}(x)$
is constant on $A_{v,u}$, and that constant is the same for every
$w\in B_{v,u}$, so we have
\[
\underset{x\in A_{v,u}}{\sum}k^{v}(x)\widetilde{\phi}(x)=\underset{x\in A_{v,u}}{\sum}c_{2}\widetilde{\phi}(x)=\underset{x\in
A_{v,u}}{\sum}k^{u}(x)\widetilde{\phi}(x)=E_{n}^{T}\cdots E_{n+k}^{T} \left( \widetilde{\phi} \right)(u)=\Lambda_{n,+k}\left( \phi \right)(u)
\]
which proves the first part of the lemma. The second part follows
by observing that $c_{1},c_{2}$ are natural constants greater than
$0$, and the same procedure can be applied to every vertex that shares
a $k$-forward path with $v$.
\end{proof}

\begin{lem} \label{lem:310}
	\label{comp_lemma}Let $G$ be a family preserving graph, and let
	$r\in\mathbb{N}$. Then for every \mbox{$k\in{\left\{-n(r),-n(r)+1,\ldots,0,1,\ldots\right\}}$} and for every $v\in S_{n(r)+k}$, either
$\underset{x\in G_{v}\cap S_{n(r)}}{\sum}\phi_{0}^{r}(x)=0$
	or $\phi_{0}^{r}|_{G_{v}\cap S_{n(r)}}\equiv c$ for some constant
	$c\in\mathbb{C}$.
\end{lem}

\begin{proof}
We prove the lemma for $k\geq0$, the proof for $k<0$ is symmetric. By Theorem \ref{thm_br}, $\phi_{0}^{r}$ is an eigenvector of the
operator $\Lambda_{n,+k}$. Assume $\Lambda_{n,+k} \left( \phi_{0}^{r} \right)=\lambda_{n,+k}\phi_{0}^{r}$
for some $\lambda_{n,+k}\in\mathbb{C}$. Let $x,y\in G_{v}\cap S_{n(r)}$.
The set $G_{v}\cap S_{n(r)}$ is exactly the set of vertices in $S_{n(r)}$
which share a $k$-forward path with $x$ (or with $y$), which
means (by Lemma \ref{lambda_const}) that $\Lambda_{n,+k}(\phi_{0}^{r})$
is constant on that set. Combining that with the fact that $\phi_{0}^{r}$
is an eigenvector of $\Lambda_{n,+k}$, we get that for every $x\in G_{v}\cap S_{n(r)}$,
\begin{equation}
\lambda_{n,+k}\phi_{0}^{r}(x)=\Lambda_{n,+k}(\phi_{0}^{r})(x)=c\underset{y\in G_{v}\cap S_{n(r)}}{\sum}\phi_{0}^{r}(y)
\end{equation}
Where $c$ is a natural number greater than $0$. If $\lambda_{n,+k}=0$,
we get $\underset{y\in G_{v}\cap S_{n(r)}}{\sum}\phi_{0}^{r}(y)=0$.
Otherwise, dividing $(3.4)$ by $\lambda_{n,+k}$ proves the result.
\end{proof}
\begin{lem}
\label{lemma_const}Let $G$ be a family preserving graph, and let
$r,k\in\mathbb{N}$. Assume that $\phi_{k}^{r}\neq0$. Then for every
$v\in S_{n(r)+k}$,  $\phi_{0}^{r}$ is constant on $G_{v}\cap S_{n(r)}$.
\end{lem}

\begin{proof}
Let $v\in S_{n(r)+k}$ be such that $\phi_{k}^{r}(v)\neq0$. By Remark
\ref{remark_Br}, this means that $\Delta_{d}^{k}(\phi_{0}^{r})(v)\neq0$,
and by Remark \ref{remark_En} and $(3.1)$, this means that $\underset{x\in G_{v}\cap S_{n(r)}}{\sum}\phi_{0}^{r}\neq0$.
By the previous lemma, we get that $\phi_{0}^{r}$ is constant on
$G_{v}\cap S_{n(r)}$. By Theorem \ref{thm_br} and the second part of Lemma \ref{lambda_const},
$\phi_{0}^{r}$ is an eigenvector of $\Lambda_{n,+k}$ with an eigenvalue
$\lambda_{n,+k}\neq0$. Let $w\in S_{n(r)+k}$ be such that $G_{w}\cap S_{n(r)}\neq G_{v}\cap S_{n(r)}$.
Assume that $\phi_{0}^{r}$ is not constant on $G_{w}\cap S_{n(r)}$,
and let $x\in G_{w}\cap S_{n(r)}$ be such that $\phi_{0}^{r}(x)\neq0$.
By Lemma \ref{comp_lemma}, $\underset{y\in G_{w}\cap S_{n(r)}}{\sum}\phi_{0}^{r}(y)=0$.
By the second part of Lemma \ref{lambda_const}, $\Lambda_{n,+k}(\phi_{0}^{r})(x)=c\underset{y\in G_{w}\cap S_{n(r)}}{\sum}\phi_{0}^{r}(y)=0$.
On the other hand, $\Lambda_{n,+k}(\phi_{0}^{r})(x)=\lambda_{n,+k}\phi_{0}^{r}(x)\neq0$,
which is a contradiction. Thus we have that $\phi_{0}^{r}$ is constant
on $G_{w}\cap S_{n(r)}$, as required.
\end{proof}
\begin{prop}
\label{lemma_same_space}Let $G$ be a family preserving graph and
$r,k\in\mathbb{N}$ such that $\phi_{k}^{r}\neq0$. then $F_{r}=F_{r,k}$.
\end{prop}

\begin{proof}
It is sufficient to show that there exists $c\in\mathbb{C}$ such
that for every $v\in S_{n(r)+k}$, $\phi_{0}^{r}|_{G_{v}\cap S_{n(r)}}\equiv c\phi_{k}^{r}(v)$,
and by Remark \ref{remark_Br} we may examine $\Delta_{d}^{k}(\phi_{0}^{r})$.
By Lemma \ref{lemma_const}, for every $v\in S_{n(r)+k}$, $\phi_{0}^{r}$
is constant on $G_{v}\cap S_{n(r)}$. Furthermore, by $(3.2)$ and
remark \ref{remark_En}, $\Delta_{d}^{k}(\phi_{0}^{r})(v)=c\underset{x\in G_{v}\cap S_{n(r)}}{\sum}\phi_{0}^{r}(v)$,
where $c$ is the number of descending paths from some $x\in G_{v}\cap S_{n(r)}$
to $v$. By spherical symmetry of $G$, for every $w\in S_{n(r)+k}$,
the number of descending paths from some $x\in G_{w}\cap S_{n(r)}$
is exactly $c$, and we get the result.
\end{proof}
%


\subsection{Some Results Regarding the Projections $(P_{n,r})_{n,r\in\mathbb{N}}$}

With the lemmas presented we are almost ready to prove Theorem \ref{thm_main}.
As a final preliminary we show that for every $r\in\mathbb{N}$, the transformation
$P_{r}$ defined by $(2.3)$ is indeed an orthogonal projection.
\begin{lem}
Let $\Gamma$ be a family preserving metric graph and let $r,k\in\mathbb{N}$.
Then for every $0 \leq t <h(\Gamma)$, either $||h_{t}^{r,k}||_{\mathbb{C}^{g_{\Gamma}(t)}}=1$
or $h_{t}^{r,k}=0$.
\end{lem}

\begin{proof}
Let $0\leq t < h(\Gamma)$. Assume that $h_{t}^{r,k}\neq0$. Then there exists some
$x\in\Gamma$ with $|x|=t$ and $\underset{v\in G_{x}\cap S_{n(r)+k}}{\sum}\phi_{0}^{r}(v)\neq0$.
This means, by Lemma \ref{lemma_const}, that $\phi_{0}^{r}$ is constant
on $G_{y}\cap S_{n(r)}$ for every $y\in\Gamma$ with $|y|=t$. This implies that $\phi_{k}^{r}$ is constant on $G_y\cap{S_{n(r)+k}}$. Now,
we have
\begin{align*}
\left\langle h_{t}^{r,k},h_{t}^{r,k}\right\rangle & =\underset{|x|=t}{\sum}h^{r,k}(x)\overline{h^{r,k}(x)}\\
 & =\underset{|x|=t}{\sum}\frac{1}{g_{n(r)+k}(t)\cdot|G_{x}\cap S_{n(r)+k}|} \left(\underset{v\in G_{x}\cap
 S_{n(r)+k}}{\sum}\phi_{k}^{r}(v) \right)\overline{\left(\underset{u\in G_{x}\cap S_{n(r)+k}}{\sum}\phi_{k}^{r}(u) \right)}=
\end{align*}
\begin{equation}
=\frac{1}{g_{n(r)+k}(t)}\underset{|x|=t}{\sum}|G_{x}\cap S_{n(r)+k}|\phi_{k}^{r}(v_{x})\overline{\phi_{k}^{r}(v_{x})}
\end{equation}

where $v_{x}$ is some representative of $G_{x}\cap S_{n(r)+k}$. Now,
note that
\[
1=\langle\phi_{k}^{r},\phi_{k}^{r}\rangle_{\ell^{2}}^2=\underset{v\in S_{n(r)+k}}{\sum}\phi_{k}^{r}(v)\overline{\phi_{k}^{r}(v)}
\]
and if we divide $S_{n(r)+k}$ to the classes of $G_{x}\cap S_{n(r)+k}$ for
$x$ with $|x|=t$, say $S_{n(r)+k}=\underset{i=1}{\overset{k}{\bigsqcup}}G_{x_{i}}\cap S_{n(r)+k}$,
then
\begin{equation}
\underset{v\in S_{n(r)+k}}{\sum}\phi_{k}^{r}(v)\overline{\phi_{k}^{r}(v)}=\underset{i=1}{\overset{k}{\sum}}|G_{x_{i}}\cap
S_{n(r)+k}|\phi_{k}^{r}(v_{x_{i}})\overline{\phi_{k}^{r}(v_{x_{i}})}
\end{equation}
Now, each summand in the right hand side of $(3.4)$ appears $g_{n(r)+k}(t)$
times in $(3.3)$, and we have
\[
\left\langle h_{t}^{r,k},h_{t}^{r,k}\right\rangle_{\mathbb{C}^{g_{\Gamma}(t)}}=\left\langle\phi_{k}^{r},\phi_{k}^{r}\right\rangle_{\ell^{2}}=1
\]
as required.
\end{proof}
\begin{lem}\label{lemma_ort_proj}
For every $r,k\in{\mathbb{N}}$, the image of $L^2(\Gamma)$ under $P_{r,k}$ (as defined by $(2.5)$) is contained in $L^2(\Gamma)$. Furthermore, $P_{r,k}$ is an orthogonal projection.
\end{lem}

\begin{proof}
The first part of the lemma follows from the fact that for every $0 \leq t <h(\Gamma)$
we have $\left \| P_{r,k}(f)_{t} \right \|_{\mathbb{C}^{g_{\Gamma}(t)}}\leq \left \| f_{t} \right \|_{\mathbb{C}^{g_{\Gamma}(t)}}$.
Additionally, for every $x\in\Gamma$ we have
\[
P_{r,k}(P_{r,k}(f))(x)=\left \langle \left \langle f_{|x|},h_{|x|}^{r,k} \right \rangle h_{|x|}^{r,k},h_{|x|}^{r,k} \right \rangle
h^{r,k}(x)=\left \langle
f_{|x|},h_{|x|}^{r,k}\right \rangle\underset{=1}{\underbrace{\left \langle h_{|x|}^{r,k},h_{|x|}^{r,k} \right
\rangle}}h^{r,k}(x)=P_{r,k}(f)(x)
\]
thus $P_{r,k}$ is indeed a projection. Finally, we need to show that
$P_{r,k}$ is self adjoint. Let $f\in L^{2}(\Gamma)$. We need to
show that $\left \langle P_{\phi}(f),f-P_{\phi}(f) \right \rangle_{L^2(\Gamma)}=0$. Indeed
\begin{align*}
\left \langle P_{r,k}(f),f-P_{r,k}(f) \right \rangle & =\int_{\Gamma}P_{r,k}(f)(x)\overline{\left(f-P_{r,k}(f)(x) \right)}d\mu(x)\\
 & =\int_{0}^{h(\Gamma)}\underset{|x|=t}{\sum}P_{r,k}(f)(x)\overline{\left(f-P_{r,k}(f)(x)\right)}dt\\
 & =\int_{0}^{h(\Gamma)}\underset{|x|=t}{\sum}\left \langle f_{t},h_{t}^{r,k} \right \rangle h^{r,k}(x)\overline{\left(f(x)-\left \langle
 f_{t},h_{t}^{r,k}\right \rangle
 h^{r,k}(x)\right )}dt\\
 & =\int_{0}^{h(\Gamma)}\langle\langle f_{t},h_{t}^{r,k}\rangle h_{t}^{r,k},f_{t}-\langle f_{t},h_{t}^{r,k}\rangle h_{t}^{r,k}\rangle dt\\
 & =\int_{0}^{h(\Gamma)}0dt=0
\end{align*}
where the next to last equality follows since $||h_{t}^{r,k}||_{\mathbb{C}^{g_{\Gamma}(t)}}=1$.
\end{proof}
%


\subsection{Proof of Theorem \ref{thm_main}}
\begin{proof}
Let $\Gamma$ be a locally balanced family preserving metric graph.

$(i)$
We first prove orthogonality. Note that for every $t\geq0$, $\left( P_{r}(f) \right)_{t}$
is the orthogonal projection of $f_{t}$ to the span of $h_{t}^{r}$.
In addition, for every $r_{1}\neq r_{2}$, the orthogonality of $\phi_{0}^{r_{1}}$
and $\phi_{0}^{r_{2}}$ implies that $h_{t}^{r_{1}}$ and $h_{t}^{r_{2}}$
are also orthogonal in $\mathbb{C}^{g_{\Gamma}(t)}$. So given $f\in F_{r_{1}}$,
$f=P_{r_{1}}(f)$, $g\in F_{r_{2}}$, $g=P_{r_{2}}(g)$, we have that
$\underset{|x|=t}{\sum}P_{r_{1}}(f)(x)\overline{P_{r_{2}}(g)(x)}=0$
which implies that $\langle f,g\rangle_{L^{2}}=0$, as required.

Now, let $f\in L^{2}(\Gamma)$ such that $supp(f)\subseteq e=(u,v)$,
$gen(e)=n$. For every $e'\in E(\Gamma)$ such that $gen(e')=gen(e)$,
define $f_{e'}$ on $e'$ to be the translation of $f$ to the edge
$e'$. That is, for $x$ on $e'$, $f_{e'}(x)=f(y)$ where $y\in e$ such that $|y|=|x|$.

Define $H_{f}$ to be the subspace of $L^{2}(\Gamma)$ spanned
by the set $\{f_{e'}\,:\,gen(e')=gen(e)\}$. Since $\Gamma$ is locally balanced, either
$$
\#\{e\in E\,:\,gen(e)=n\}=\#\{v\in V\,:\,gen(v)=n\}
$$
or
$$
\#\{e\in E\,:\,gen(e)=n\}=\#\{v\in V\,:\,gen(v)=n+1\}.
$$
Assume first that $\#\{e\in E\,:\,gen(e)=n\}=\#\{v\in V\,:\,gen(v)=n\}$ and let
\mbox{$\{\phi_{1},\ldots,\phi_{|S_{n}|}\}\subseteq\{\phi_{j}^{r}\,:\,j,r\in\mathbb{N}\}$}
be a basis of $\ell^{2}(S_{n})$. For every $1\leq i\leq|S_{n}|$,
define $g^{i}\in H_{f}$ in the following way:
\[
g^{i}(x)=f_{e_x}(x)\phi_{i}(v_{x})
\]
where $e_{x}$ is the edge containing $x$, and $v_{x}=i(e_x)$ is the initial vertex of $e_x$. In other words,
$$
g^{i}=\sum_{gen(e')=n}\phi_{i}(i(e)) f_{e'},
$$
so clearly, $g^i \in H_f$. Now, we have (note that if $gen(e')=n$ then $|G_x\cap S_n|=1$)
\begin{align*}
P_{\phi_{i}}(g^{i})(x) & =\left \langle g_{|x|}^{i},h_{|x|}^{\phi_{i}} \right \rangle h^{\phi_{i}}(x)=\\
 & =\frac{1}{g_n(|x|)}\left(\underset{|y|=|x|}{\sum}f_{e_{y}}(y)\phi_{i}(v_y)\overline{\phi_{i}(v_y)} \right)\phi_{i}(v_{x})=\\
 & =f_{e_x}(x)\left( \underset{v\in S_{n}}{\sum}\phi_{i}(v)\overline{\phi_{i}(v)} \right)\phi_{i}(v_{x})=\\
 & =f_{e_x}(x)\phi^{i}(v_{x})=g^{i}(x)
\end{align*}
where the third equality follows from the fact that $|x|=|y|\Rightarrow f_{e_x}(x)=f_{e_y}(y)$,
and the fourth from the fact that $||\phi_{i}||_{\ell^{2}}=1$. This implies that $g^i \in \textrm{Im} \left(P_{\phi^i} \right)$
and so that $g^i \in {\oplus}F_{r}$.
Moreover, by the independence of the $\phi^i$, it is clear that the functions $\{g^{1},\ldots,g^{|S_{n}|}\}$ are linearly
independent. Thus, by dimension considerations it follows that $H_{f}$ is spanned by $\{g^{1},\ldots,g^{|S_{n}|}\}$. This implies that $H_f
\subseteq \oplus F_r$, and in particular $f\in\underset{r\in\mathbb{N}}{\oplus}F_{r}$. Since linear combinations of such $f$'s are dense in
$L^2(\Gamma)$ we obtain $L^2(\Gamma)=\underset{r\in\mathbb{N}}{\oplus}F_{r}$ as required.

The proof in the case that $\#\{e\in E\,:\,gen(e)=n\}=\#\{v\in V\,:\,gen(v)=n+1\}$ is similar.

$(ii)$ First, since differentiation is a local action and $h^{r}$
is constant on edges, it is easy to see that $\Delta P_{r}=P_{r}\Delta$
on $D(\Delta)$. Moreover, for the same reason differentiability properties of $\varphi$ on the edges are unchanged by $P_r$.
Therefore to show that $P_r(D(\Delta))\subseteq D(\Delta)$ we only need to check the gluing conditions at the vertices. Thus let
$\varphi\in{D(\Delta)}$, $f=P_r(\varphi)$, and let $v\in S_{n(r)+k}$, where $k\in\{-n(r)+1,-n(r),\ldots,0,1,\ldots\}$. We divide into cases.

Assume first that $\underset{u\in{G_w\cap{S_{n(r)}}}}{\sum}\phi_{0}^{r}(u)=0$ for every $w\in{S_{n(r)+k}}$. In that case $k\neq0$. We treat
the case that $k\geq1$. The proof for $k\leq-1$ is symmetric. Note that for $x\in G_v$ such that $|x|>|v|$, ${G_v\cap{S_{n(r)}}}\subseteq
G_x\cap{S_{n(r)}}$. Thus, by Lemma \ref{lem:310} and the definition of $P_r$, $f$ vanishes on the set $\left\{x\in{\Gamma}\ |\ |x|\geq|v|\right\}$.

In order to check the gluing conditions, we first check continuity at $v$.
We need to verify that $\underset{x\nearrow v}{lim}f(x)=0$. For convenience,
we denote $g_{n(r)}(|x|)$  by $g_{n(r)}$, and $|G_{x}\cap S_{n(r)}|$ by
$w_{n(r)}$ since they are constant along segments between vertices.
\begin{align*}
\underset{x\nearrow v}{lim}f(x) & =\underset{x\nearrow v}{lim}\left(\underset{|y|=|x|}{\sum}\varphi(y)\cdot\frac{1}{\sqrt{g_{n(r)}\cdot
w_{n(r)}}}\underset{w\in G_{y}\cap S_{n(r)}}{\sum}\phi_{0}^{r}(w)\right)\frac{1}{\sqrt{g_{n(r)}\cdot w_{n(r)}}}\underset{u\in G_{x}\cap
S_{n(r)}}{\sum}\phi_{0}^{r}(u)=\\
& =\underset{x\nearrow v}{lim}\frac{1}{g_{n(r)}\cdot w_{n(r)}}\left(\underset{|y|=|x|}{\sum}\varphi(y)\underset{w\in G_{y}\cap
S_{n(r)}}{\sum}\phi_{0}^{r}(w)\right)\underset{u\in G_{x}\cap S_{n(r)}}{\sum}\phi_{0}^{r}(u)=\\
& =\underset{x\nearrow v}{lim}\frac{1}{g_{n(r)}\cdot w_{n(r)}}\left(\underset{s\in S_{n(r)+k}}{\sum}\underset{|y|=|x|}{\underset{y\in
G_{s}}{\sum}}\varphi(y)\underset{w\in G_{y}\cap S_{n(r)}}{\sum}\phi_{0}^{r}(w)\right)\underset{u\in G_{x}\cap
S_{n(r)}}{\sum}\phi_{0}^{r}(u)=\\
& =\underset{x\nearrow v}{lim}\frac{1}{g_{n(r)}\cdot w_{n(r)}}\left(\underset{s\in S_{n(r)+k}}{\sum}\varphi(s)\underset{=0}{\underbrace{\underset{|y|=|x|}{\underset{y\in
G_{s}}{\sum}}\underset{w\in G_{y}\cap S_{n(r)}}{\sum}\phi_{0}^{r}(w)}}\right)\underset{u\in G_{x}\cap S_{n(r)}}{\sum}\phi_{0}^{r}(u)=\\
& =0
\end{align*}
where the next to last equality follows from continuity of $\varphi$, and
the rest follow from changing the order of summation. As for the derivative
matching condition, let $e\in{E(\Gamma)}$ such that $t(e)=v$. Note that $h^r$ is constant on $e$, so for $x(e)\in{e}$ we denote
$h^r(x(e))=c_e$.
Now, for every $x\in{e}$ $(f|_e)(x)=c_e\left< \varphi_{|x|},h_{|x|}^{r} \right>$. Thus, we may write
\[
\underset{e:t(e)=v}{\sum}(f|_{e})'(v)=\left(\left\langle \varphi_{|v|},h_{|v|}^{r} \right \rangle
\right)'\cdot\sum_{e:t(e)=v}c_e
\]
Now note that
\[
\sum_{e:t(e)=v}c_e=0
\]
as it is some multiple of $\underset{u\in G_{v}\cap S_{n(r)}}{\sum}\phi_{0}^{r}(u)$.
This is a consequence of the fact that $G_v\cap{S_n(r)}=\underset{e:t(e)=v}{\bigcup}G_{x(e)}\cap{S_n(r)}$ and for every
$e_1,e_2\in{E(\Gamma)})$ such that $t(e_1)=t(e_2)=v$,
\begin{align*}
 |\left\{e:G_{x(e)}\cap{S_{n(r)}}=G_{x(e_1)}\cap{S_{n(r)}}, t(e)=v\right\}|=|\left\{e:G_{x(e)}\cap{S_{n(r)}}=G_{x(e_2)}\cap{S_{n)r}},
 t(e)=v\right\}|
 \end{align*}
 which in turn follows from Corollary \ref{same_num_paths}.
Thus, by our assumption, the above sum is $0$ as required.

We now treat the case in which there exists some $w\in{S_{n(r)+k}}$ for which $\underset{u\in{G_w\cap{S_{n(r)}}}}{\sum}\phi_{0}^{r}(u)\neq0$.
This is impossible for $k<0$ since $H_r$ is orthogonal to $\ell^2(S_n)$ for every $n<n(r)$, and the above sum is a
multiple of $\left< \Delta_d^{|k|}\phi_{0}^{r},\delta_w \right>$. Thus $k\geq0$.
The fact that the above sum is not $0$ implies that $\phi_{k}^{r}\neq0$. Thus, by Lemma \ref{lemma_same_space}, $F_r=F_{r,k}$. By the
uniqueness of the orthogonal projection and by Lemma \ref{lemma_ort_proj}, we may write
$f=P_{r,k}(\varphi)$.
For sufficiently small $\epsilon>0$, for every $x\in\Gamma$
with $|x|=|v|\pm\epsilon$ it holds that $|G_x\cap{S_{n(r)+k}}|=1$, thus
\[
h^{\phi_{k}^{r}}(x)=\frac{1}{\sqrt{g_{n(r)+k}(|x|)}}\phi_{k}^{r}(v_{x})
\]
where $v_x \in S_{n(r)+k}$ and $x$ lies on the edge $e$ for which either $i(e)=v_{x}$ or $t(e)=v_{x}$.
Thus, for $v\in S_{n(r)+k}$ and $x\in\Gamma$ which lies on an edge
$(u,v)\in E(\Gamma)$, we have $v_x=v$ and so
\begin{align*}
f(x) &
=P_{r,k}(\varphi)(x)=\left(\underset{|y|=|x|}{\sum}\varphi(y)\frac{1}{\sqrt{g_{n(r)+k}(|y|)}}\phi_{k}^{r}(v_{y})\right)\frac{1}{\sqrt{g_{n(r)+k}(|x|)}}\phi_{k}^{r}(v)=\\
& =\left(\underset{w\in S_{n(r)+k}}{\sum}\phi_{k}^{r}(w)\underset{|y|=|x|, y\in
G_{w}}{\sum}\frac{\varphi(y)}{g_{n(r)+k}(|y|)}\right)\phi_{k}^{r}(v)
\end{align*}
Now, when $x$ tends to $v$, $\varphi(y)$ tends to $\varphi(w)$ and so
\[
\underset{|y|=|x|, y\in
G_{w}}{\sum}\frac{\varphi(y)}{g_{n(r)+k}(|y|)}\rightarrow \varphi(w).
\]
It follows that
\[
P_{r,k}(\varphi)(x)\underset{x\rightarrow{v}}{\longrightarrow}\phi_{k}^{r}(v)\underset{w\in S_{n(r)+k}}{\sum}\phi_{k}^{r}(w)\varphi(w)
\]
which means that $f$ is continuous.

We now prove that the derivative matching
condition holds. In order to do so, we divide the edges of which $v$
is a part into two sets. The set of edges which terminate in $v$
will be denoted as $E_{t}$, and the set of edges emanating from $v$
	will be denoted as $E_{s}$. We also denote $g_{n(r)+k}(|x|)\coloneqq b_{n(r)+k}^{in}$
for $|u|<|x|<|v|$ with $u\in{S_{n(r)+k-1})}$ (note that $b_{n(r)+k}^{in}$ is the number of edges terminating in $v$), and $g_{n(r)+k}(|x|)\coloneqq b_{n(r)+k}^{out}$ for
$|u|>|x|>|v|$ with $u\in{S_{n(r)+k+1}}$ ($b_{n(r)+k}^{out}$ is the number of edges emanating from $v$). Finally, for every edge $e\in{E(\Gamma)}$ denote by $x(e)$ the selection of some point $x\in{e}$. Now, we calculate the sum
$\underset{u\sim v}{\sum}(P_{r,k}(\varphi)|_{(u,v)})(x(u,v))$:
\begin{align*}
\underset{u\sim v}{\sum}\left(P_{r,k}(\varphi)|_{(u,v)}\right)\left(x(u,v)\right) & =\underset{e\in
E_{t}}{\sum}\left(P_{r,k}(\varphi)|_{e}\right)\left(x(e)\right)+\underset{e\in
E_{s}}{\sum}\left(P_{r,k}(\varphi)|_{e}\right)\left(x(e)\right)=\\
& =\underset{e\in E_{t}}{\sum}\left(\underset{w\in S_{n(r)+k}}{\sum}\underset{|y|=|x(e)|}{\underset{y\in
G_{w}}{\sum}}\frac{1}{\sqrt{b_{n(r)+k}^{in}}}\phi_{k}^{r}(w)\varphi(y)\right)\frac{1}{\sqrt{b_{n(r)+k}^{in}}}\phi_{k}^{r}(v)+\\
& +\underset{e\in E_{s}}{\sum}\left(\underset{w\in S_{n(r)+k}}{\sum}\underset{|y|=|x(e)|}{\underset{y\in
G_{w}}{\sum}}\frac{1}{\sqrt{b_{n(r)+k}^{out}}}\phi_{k}^{r}(w)\varphi(y)\right)\frac{1}{\sqrt{b_{n(r)+k}^{out}}}\phi_{k}^{r}(v)
\end{align*}

Note that if we pick $|x(e_{1})|=|x(e_{2})|=t_{1}$ for every $e_{1},e_{2}\in E_{t}$
and $|x(e_{3})|=|x(e_{4})|=t_{2}$ for every $e_{3},e_{4}\in E_{s}$,
then we may write the above sum as
\begin{align*}
& b_{n(r)+k}^{in}\left(\underset{w\in S_{n(r)+k}}{\sum}\underset{|y|=t_{1}}{\underset{y\in
G_{w}}{\sum}}\frac{1}{\sqrt{b_{n(r)+k}^{in}}}\phi_{k}^{r}(w)\varphi(y)\right)\frac{1}{\sqrt{b_{n(r)+k}^{in}}}\phi_{k}^{r}(v)\\
& \quad + b_{n(r)+k}^{out}\left(\underset{w\in S_{n(r)+k}}{\sum}\underset{|y|=t_{2}}{\underset{y\in
G_{w}}{\sum}}\frac{1}{\sqrt{b_{n(r)+k}^{out}}}\phi_{k}^{r}(w)\varphi(y)\right)\frac{1}{\sqrt{b_{n(r)+k}^{out}}}\phi_{k}^{r}(v)=\\
& =\left(\underset{w\in S_{n(r)+k}}{\sum}\phi_{k}^{r}(w)\underset{|y|=t_{1}}{\underset{y\in G_{w}}{\sum}}\varphi(y)\right)\phi_{k}^{r}(v)+\\
& \quad +\left(\underset{w\in S_{n(r)+k}}{\sum}\phi_{k}^{r}(w)\underset{|y|=t_{2}}{\underset{y\in
G_{w}}{\sum}}\varphi(y)\right)\phi_{k}^{r}(v) \\
& =\phi_{k}^{r}(v)\underset{w\in S_{n(r)+k}}{\sum}\phi_{k}^{r}(w)\left(\underset{|y|=t_{1}}{\underset{y\in
G_{w}}{\sum}}\varphi(y)+\underset{|y|=t_{2}}{\underset{y\in G_{w}}{\sum}}\varphi(y)\right)
\end{align*}
Since $\varphi \in D(\Delta)$, by taking the limit $|t_1|$ and $|t_2|$ to $|v|$ and differentiating, we see that the derivative matching
condition holds.
\end{proof}
The proof of The proof of the second part of Theorem \ref{thm_main}, along with a description of the spaces $\left(F_r\right)_{r\in{\mathbb{N}}}$, is given in Section 4. second part of Theorem \ref{thm_main} is given in Section 4.


\section{Decomposing $\Delta$}
Let $\Gamma$ be a family preserving metric graph. In this section, we describe an algorithm with which one can produce the one dimensional components of the Laplacian on $\Gamma$. We will also demonstrate each step on an example, denoted by $\widetilde{\Gamma}$ and presented in Figure \ref{ex_present}, which will accompany us throughout the section. The analysis of two more examples (trees and antitrees) is given in Section 5.
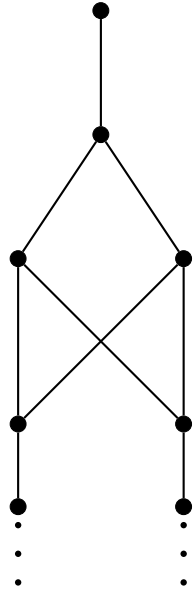
\begin{figure}
	\begin{center}
		\begin{tikzpicture}[scale=1.1]
		\vertex(o) at (3,1.5) {};
		\vertex(v) at (3,0) {};
		\vertex(w1) at (2,-1.5) {};
		\vertex(w2) at (4,-1.5) {};
		\vertex(x1) at (2,-3.5) {};
		\vertex(x2) at (4, -3.5) {};
		\vertex(x3) at (2, -4.5) {};
		\vertex(x4) at (4, -4.5) {};
		\node[vertex, draw=none, fill=white](dots1) at (2, -5.5) {};
		\node[vertex, draw=none, fill=white](dots2) at (4, -5.5) {};
		\path (x3) -- (dots1) node [black, font=\Huge, midway, sloped] {$\dots$};
		\path (x4) -- (dots2) node [black, font=\Huge, midway, sloped] {$\dots$};
		\Edge(o)(v)
		\Edge(v)(w1)
		\Edge(v)(w2)
		\Edge(w1)(x1)
		\Edge(w1)(x2)
		\Edge(w2)(x1)
		\Edge(w2)(x2)
		\Edge(x2)(x4)
		\Edge(x1)(x3);
		\end{tikzpicture}
	\end{center}
\captionof{figure}{The graph $\widetilde{\Gamma}$. The dots signify some family preserving continuation of the graph. Note that the vertices are denoted by black discs, so there is no vertex where the drawings of the edges intersect.}
\label{ex_present}
\end{figure}

\subsection{Step 0 - Generating the equivalent locally balanced graph}
In the case that $\Gamma$ is not locally balanced, one should turn it into a locally balanced graph by applying the following procedure: For every $n\in{\mathbb{N}}$, check whether $\#\left\{e\in{E(\Gamma)}\,|\, gen(e)=n\right\}=|S_{n-1}|$ or $\#\left\{e\in{E(\Gamma)}\,|\, gen(e)=n\right\}=|S_{n}|$. If not, add dummy vertices on all of the edges of generation $n$. The vertices are all located at an equal distance from $S_{n-1}$, and kirchhoff boundary conditions are attached to them. This process is illustrated in Figure \ref{ex_loc_bal}. In the appendix we show that this procedure can be carried out to any family preserving metric graph, and that the resulting operator is unitarily equivalent to the original one.
\begin{figure}
	\begin{center}
		\begin{tikzpicture}[scale=1.1]
		\vertex(o) at (3,1.5) {};
		\vertex(v) at (3,0) {};
		\vertex(w1) at (2,-1.5) {};
		\vertex(w2) at (4,-1.5) {};
		\vertex(x1) at (2,-3.5) {};
		\vertex(x2) at (4, -3.5) {};
		\vertex(x3) at (2, -4.5) {};
		\vertex(x4) at (4, -4.5) {};
		\node[vertex, draw=none, fill=white](dots1) at (2, -5.5) {};
		\node[vertex, draw=none, fill=white](dots2) at (4, -5.5) {};
		\path (x3) -- (dots1) node [black, font=\Huge, midway, sloped] {$\dots$};
		\path (x4) -- (dots2) node [black, font=\Huge, midway, sloped] {$\dots$};
		\Edge(o)(v)
		\Edge(v)(w1)
		\Edge(v)(w2)
		\Edge(w1)(x1)
		\Edge(w1)(x2)
		\Edge(w2)(x1)
		\Edge(w2)(x2)
		\Edge(x2)(x4)
		\Edge(x1)(x3);
		\draw[->] (5, -2.4) -- (7, -2.4);
		\vertex(o') at (9, 1.5) {};
		\draw (9.3, 1.5) node{$o$};
		\vertex(v') at (9, 0) {};
		\draw (9.3, 0) node{$v$};
		\vertex(w1') at (8, -1.5) {};
		\draw (7.7, -1.5) node{$w_1$};
		\vertex(w2') at (10, -1.5){};
		\draw (10.3, -1.5) node{$w_2$};
		\vertex(u1') at (8, -2.1666) {};
		\draw (7.7, -2.1666) node{$u_1$};
		\vertex(u2') at (8.6666, -2.1666) {};
		\draw (8.3666, -2.1666) node{$u_2$};
		\vertex(u3') at (9.3333, -2.1666) {};
		\draw (9.6333, -2.1666) node{$u_3$};
		\vertex(u4') at (10, -2.1666) {};
		\draw (10.3, -2.1666) node{$u_4$};
		\vertex(x1') at (8, -3.5) {};
		\vertex(x2') at (10, -3.5) {};
		\vertex(x3') at (8, -4.5) {};
		\vertex(x4') at (10, -4.5) {};
		\node[vertex, draw=none, fill=white](dots1') at (8, -5.5) {};
		\node[vertex, draw=none, fill=white](dots2') at (10, -5.5) {};
		\path (x3') -- (dots1') node [black, font=\Huge, midway, sloped] {$\dots$};
		\path (x4') -- (dots2') node [black, font=\Huge, midway, sloped] {$\dots$};
		\Edge(o')(v')
		\Edge(v')(w1')
		\Edge(v')(w2')
		\Edge(w1')(u1')
		\Edge(w1')(u2')
		\Edge(w2')(u3')
		\Edge(w2')(u4')
		\Edge(u1')(x1')
		\Edge(u2')(x2')
		\Edge(u3')(x1')
		\Edge(u4')(x2')
		\Edge(x1')(x3')
		\Edge(x2')(x4');
		\end{tikzpicture}
	\end{center}
\captionof{figure}{Transforming $\widetilde{\Gamma}$ into a locally balanced graph.}
\label{ex_loc_bal}
\end{figure}
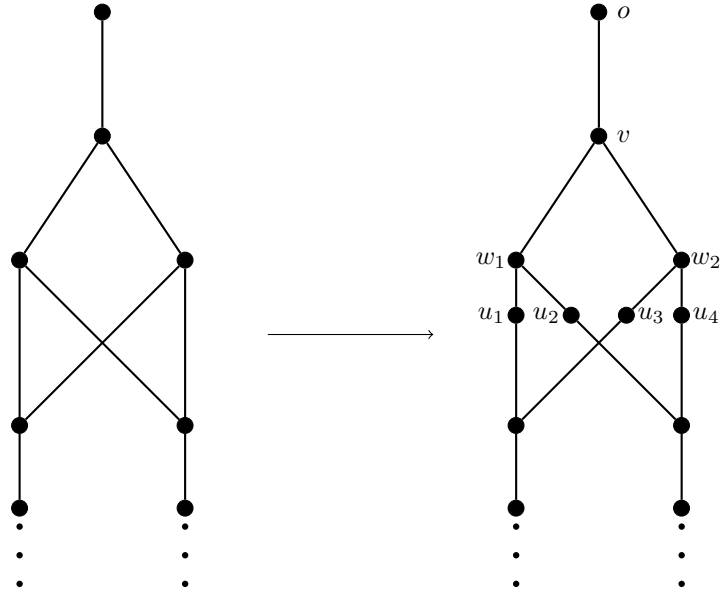

\subsection{Step 1 - Obtaining the discrete basis}
In this step we consider the discrete graph $G_{\Gamma}$. The purpose of this step is to extract the elements  $\left(\phi_{0}^{r}\right)_{r\in{\mathbb{N}}}$ from the basis given by the process described in the proof of Theorem \ref{thm_br} (\cite{BK}). Note that due to Proposition \ref{lemma_same_space} the other elements of the basis are irrelevant to the construction of $\left(F_r\right)_{r\in\mathbb{N}}$. We now provide a short description of this process. For simplicity, for a compactly supported $f\in{\ell^2(G_{\Gamma})}$ we denote $span\{f,\Delta_d\}\coloneqq\overline{span\left\{f,\Delta_{d}f,\Delta_{d}^2f,\ldots\right\}}$, where $\Delta_d$ is the discrete Laplacian acting on $\ell^2(G_{\Gamma})$. Note that for every $n\in{\mathbb{N}}$, $\Delta_{d}^{n}(f)$ is also compactly supported and thus is in the domain of $\Delta_d$
\begin{itemize}
	\item Define $\phi_{0}^{0}\coloneqq{\delta_o}$, where $o$ is the root of $G_{\Gamma}$. Also, define $W_0\coloneqq sp\left\{\delta_o,\Delta_d\right\}$.
	\item Now, let $n_1\in{\mathbb{N}}$ be minimal such that $\ell^2(S_{n_1})\nsubseteq W_0$. Let $\left\{\phi_{0}^{1},\ldots,\phi_{0}^{r_1}\right\}$ be an orthonormal basis of $\left(\ell^2(S_{n_1})\cap{W_0}\right)^{\perp}$ such that for every $1\leq i\leq r_1$ and every $j\geq 0$, $\phi_0^i$ commutes with $\Lambda_{n,+j}$ (the existence of such a basis when the graph is family preserving is proven in \cite{BK}). Define $W_1\coloneqq W_0\oplus \left(\underset{k=1,\ldots ,r_1}{\oplus}sp\left\{\phi_{0}^{r_k},\Delta_d\right\}\right)$.
	\item Proceed inductively.
\end{itemize}
With the labeling presented in Figure \ref{ex_loc_bal}, the first elements of the basis in our example are:
\begin{align*}
\phi_{0}^{0}&=\delta_0 \\
\phi_{0}^{1}&=\frac{\delta_{w_1}-\delta_{w_2}}{\sqrt{2}} \\
\phi_{0}^{2}&=\frac{\delta_{u_1}-\delta_{u_2}+\delta_{u_3}-\delta_{u_4}}{2}\\ \phi_{0}^{3}&=\frac{\delta_{u_1}-\delta_{u_2}-\delta_{u_3}+\delta_{u_4}}{2}
\end{align*}
Note that if the graph is continued by straight lines, these are the only elements in the basis.

Having obtained the elements $\left(\phi_{0}^{r}\right)_{r\in{\mathbb{N}}}$, the next step is to determine the spaces generated by them, as described in Section 2.2.
\subsection{Step 2 - Determining the spaces $\left(F_r\right)_{r\in{\mathbb{N}}}$}
Let $r\in\mathbb{N}$. Recall that for $x\in{\Gamma}$,
\begin{align*}
h^r(x)=\frac{1}{\sqrt{g_{n(r)}(|x|)\cdot|G_{x}\cap S_{n(r)}|}}\underset{v\in G_{x}\cap S_{n(r)}}{\sum}\phi_{0}^{r}(v)
\end{align*}
where $S_{n(r)}$ is the sphere on which $\phi_{0}^{r}$ is supported. Let
\begin{equation} \nonumber
a_r=\inf \left \{|x| \mid h^r(x)\neq 0 \right \}
\end{equation}
and
\begin{equation} \nonumber
b_r=\sup \left \{|x| \mid h^r(x)\neq 0 \right \}.
\end{equation}
Now let $f\in F_{r}$ (so that $f=P_{r}(f)$), and let $\varphi_{f}(x)=\left\langle
f_{|x|},h_{|x|}^{r}\right\rangle_{\mathbb{C}^{g_\Gamma(|x|)}}$. Then for every $f\in F_{r}$,
\begin{equation} \label{eq:f_mult}
f=\varphi_f(x)h^r(x).
\end{equation}
Clearly, $supp(f)\subseteq\left\{x\in\Gamma\,:\,|x|\in[a_r,b_r]\right\}$\footnote{If $b_r=\infty$, one should replace $[a_r,b_r]$ in $[a_r,b_r)$ etc.} and $|x|=|y|\Rightarrow\varphi_{f}(x)=\varphi_{f}(y)$.

Recall that $t_k$ is the distance in $\Gamma$ of a vertex $\in S_k$ from the root $o$.
Let $k\in\mathbb{N}$ be such that $a_{r}=t_{k}$. If $b_{r}=\infty$, let $l=\infty$, and otherwise let $l$
be the maximal $j\in\mathbb{N}$ such that $t_{j}\leq b_{r}$. Denote $A_{r}=\{k,k+1,\ldots,(l-1)\}$.
By the fact that $supp(\phi_{k}^{r})\subseteq{S_{n(r)+k}}$, one can conclude that for every $n<n(r)$, and every $v\in{S_n}$,
$\underset{u\in{G_v\cap{S_{n(r)}}}}{\sum}\phi_{0}^{r}(v)=0$. Taking this and the definition of $P_r$ into consideration, we conclude that
\mbox{$A_r=\{n(r)-1,n(r),\ldots,l-1\}$}.
\begin{claim}
	\label{const_on_sj}
	For $n(r)\leq{k}\in{A_r}$, and for $v\in{S_k}$, $\phi_{0}^{r}$ is constant on $G_v\cap{S_{n(r)}}$. For $v\in{S_{n(r)-1}}$,
	$\underset{u\in{G_v\cap{S_{n(r)}}}}{\sum}\phi_{0}^{r}(u)=0$.
\end{claim}
\begin{proof}
	The second part of the claim follows from the fact that $H_r$ is orthogonal to $\ell^2(S_{n(r)-1})$. As for the first part, it is enough
	to show that there exists $v\in{S_k}$ such that $\underset{u\in G_{v}\cap S_{n(r)}}{\sum}\phi_{0}^{r}(u)\neq 0$. Indeed, if this is true, then
	$\phi_{0}^{r}$ is an eigenvector of $\Lambda_{n(r),+(n(r)-k)}$ with an eigenvalue $\lambda \neq 0$. Now, if there exists some $w$ such that
	$\phi_{0}^{r}$ is not constant on $G_w\cap{S_{n(r)}}$, then by Lemma \ref{lem:310},
	$\underset{u\in{G_w\cap{S_{n(r)}}}}{\sum}\phi_{0}^{r}(u)=0$. This means that $\phi_{0}^{r}$ is an eigenvector of $\Lambda_{n(r),+(n(r)-k)}$
	with
	$0$ as an eigenvalue, which is a contradiction. Now assume that for every $v\in{S_k}$,
	$\underset{w\in{G_v\cap{S_{n(r)}}}}{\sum}\phi_{0}^{r}(w)=0$. Combining Lemma $\ref{lem:310}$ and the fact that for $x,y\in\Gamma$ such that
	$|x|>|y|>t_{n(r)}$ it holds that $G_y\cap{S_{n(r)}}\subseteq{G_x\cap{S_{n(r)}}}$, we conclude that for every $x\in{\Gamma}$ with $|x|>t_k$,
	$h^r(x)=0$. This contradicts the fact that $k\in{A_r}$.
\end{proof}
By \eqref{eq:f_mult}, the fact that $\varphi_f$ is symmetric and the fact that $h^r$ is constant along edges, we have that for every $f\in F_r$ and $e_1, e_2\in E(\Gamma)$ such that $gen(e_1)=gen(e_2)$, $f|_{e_1}$ and $f|_{e_2}$ are multiples of each other. In other words, the space $F_r$ is determined by $a_r, b_r$ and scalar coefficients attached to the graph's edges. The scalar attached to an edge $e$ is the value $h^r$ takes on some $x\in e$, and by Claim \ref{const_on_sj}, if we choose some $v\in G_x\cap S_{n(r)}$, then $h^r(y)=\sqrt{\frac{|G_x\cap S_{n(r)}|}{g_{n(r)}(|x|)}}\phi_{0}^{r}(v)$ for every $y\in e$. This description of $F_r$ is illustrated on $\widetilde{\Gamma}$ (for $r=1,2$) in Figure \ref{ex_Fr}.
\begin{figure}
	\begin{center}
		\begin{tikzpicture}[scale=1.35]
		\draw (3, 2) node{$r=1$};
		\vertex(o) at (3,1.5) {};
		\vertex(v) at (3,0) {};
		\draw[->] (3,0) -- (4.55,0);
		\draw (4.7, 0) node{$a_1$};
		\vertex(w1) at (2,-1.5) {};
		\draw (2.1, -0.75) node {$\frac{1}{\sqrt{2}}$};
		\vertex(w2) at (4,-1.5) {};
		\draw (3.9, -0.75) node {$\frac{-1}{\sqrt{2}}$};
		\vertex(u1) at (2, -2.1666) {};
		\draw (1.75, -1.85) node{$\frac{1}{2}$};
		\vertex(u2) at (2.6666, -2.1666) {};
		\draw (2.5666, -1.75) node{$\frac{1}{2}$};
		\vertex(u3) at (3.3333, -2.1666) {};
		\draw (3.4333, -1.75) node{$-\frac{1}{2}$};
		\vertex(u4) at (4, -2.1666) {};
		\draw (4.25, -1.85) node{$-\frac{1}{2}$};
		\vertex(x1) at (2,-3.5) {};
		\draw (1.75, -2.8) node{$\frac{1}{2}$};
		\draw (2.37, -2.8) node{$-\frac{1}{2}$};
		\vertex(x2) at (4, -3.5) {};
		\draw (3.6, -2.8) node{$\frac{1}{2}$};
		\draw (4.25, -2.8) node{$-\frac{1}{2}$};
		\draw[->] (4, -3.5) -- (4.55, -3.5);
		\draw (4.7, -3.5) node{$b_1$};
		\vertex(x3) at (2, -4.5) {};
		\vertex(x4) at (4, -4.5) {};
		\node[vertex, draw=none, fill=white](dots1) at (2, -5.5) {};
		\node[vertex, draw=none, fill=white](dots2) at (4, -5.5) {};
		\path (x3) -- (dots1) node [black, font=\Huge, midway, sloped] {$\dots$};
		\path (x4) -- (dots2) node [black, font=\Huge, midway, sloped] {$\dots$};
		\Edge(o)(v)
		\Edge(v)(w1)
		\Edge(v)(w2)
		\Edge(w1)(x1)
		\Edge(w1)(x2)
		\Edge(w2)(x1)
		\Edge(w2)(x2)
		\Edge(x2)(x4)
		\Edge(x1)(x3);
		\draw (9, 2) node{$r=2$};
		\vertex(o') at (9, 1.5) {};
		\vertex(v') at (9, 0) {};
		\vertex(w1') at (8, -1.5) {};
		\vertex(w2') at (10, -1.5){};
		\draw[->] (10, -1.5) -- (11.55, -1.5);
		\draw (11.7, -1.5) node{$a_2$};
		\vertex(u1') at (8, -2.1666) {};
		\draw (7.8, -1.85) node{$\frac{1}{2}$};
		\vertex(u2') at (8.6666, -2.1666) {};
		\draw (8.5, -1.75) node{$-\frac{1}{2}$};
		\vertex(u3') at (9.3333, -2.1666) {};
		\draw (9.4666, -1.75) node{$\frac{1}{2}$};
		\vertex(u4') at (10, -2.1666) {};
		\draw (10.2, -1.85) node{$-\frac{1}{2}$};
		\vertex(x1') at (8, -3.5) {};
		\draw (7.8, -2.8) node{$\frac{1}{2}$};
		\draw (8.42, -2.8) node {$\frac{1}{2}$};
		\vertex(x2') at (10, -3.5) {};
		\draw (9.6, -2.8) node {$-\frac{1}{2}$};
		\draw (10.2, -2.8) node {$-\frac{1}{2}$};
		\vertex(x3') at (8, -4.5) {};
		\draw (7.8, -4) node {$1$};
		\vertex(x4') at (10, -4.5) {};
		\draw (10.2, -4) node {$-1$};
		\node[vertex, draw=none, fill=white](dots1') at (8, -5.5) {};
		\node[vertex, draw=none, fill=white](dots2') at (10, -5.5) {};
		\path (x3') -- (dots1') node [black, font=\Huge, midway, sloped] {$\dots$};
		\path (x4') -- (dots2') node [black, font=\Huge, midway, sloped] {$\dots$};
		\Edge(o')(v')
		\Edge(v')(w1')
		\Edge(v')(w2')
		\Edge(w1')(u1')
		\Edge(w1')(u2')
		\Edge(w2')(u3')
		\Edge(w2')(u4')
		\Edge(u1')(x1')
		\Edge(u2')(x2')
		\Edge(u3')(x1')
		\Edge(u4')(x2')
		\Edge(x1')(x3')
		\Edge(x2')(x4');
		\draw (6, -5.5) node{$r=3$};
		\vertex (o'') at (6, -6) {};
		\vertex (v'') at (6, -7.5) {};
		\vertex (w1'') at (5, -9) {};
		\vertex (w2'') at (7, -9) {};
		\draw[->] (7, -9) -- (8.55, -9);
		\draw (8.7, -9) node{$a_3$};
		\vertex (u1'') at (5, -9.6666) {};
		\draw (4.8, -9.35) node{$\frac{1}{2}$};
		\vertex (u2'') at (5.6666, -9.6666) {};
		\draw (5.5, -9.25) node{$-\frac{1}{2}$};
		\vertex (u3'') at (6.3333, -9.6666) {};
		\draw (6.4, -9.25) node{$-\frac{1}{2}$};
		\vertex (u4'') at (7, -9.6666) {};
		\draw (7.2, -9.35) node{$\frac{1}{2}$};
		\vertex (x1'') at (5, -11) {};
		\draw (4.8, -10.3) node {$\frac{1}{2}$};
		\draw (5.38, -10.3) node{$-\frac{1}{2}$};
		\vertex (x2'') at (7, -11) {};
		\draw[->] (7, -11) -- (8.55, -11);
		\draw (8.7, -11) node{$b_3$};
		\draw (6.52, -10.3) node{$-\frac{1}{2}$};
		\draw (7.2, -10.3) node{$\frac{1}{2}$};
		\vertex (x3'') at (5, -12) {};
		\vertex (x4'') at (7, -12) {};
		\node[vertex, draw=none, fill=white](dots1'') at (5, -13) {};
		\path (x3'') -- (dots1'') node [black, font=\Huge, midway, sloped] {$\dots$};
		\node[vertex, draw=none, fill=white](dots2'') at (7, -13) {};
		\path (x4'') -- (dots2'') node [black, font=\Huge, midway, sloped] {$\dots$};
		\Edge(o'')(v'')
		\Edge(v'')(w1'')
		\Edge(v'')(w2'')
		\Edge(w1'')(u1'')
		\Edge(w1'')(u2'')
		\Edge(w2'')(u3'')
		\Edge(w2'')(u4'')
		\Edge(u1'')(x1'')
		\Edge(u2'')(x2'')
		\Edge(u3'')(x1'')
		\Edge(u4'')(x2'')
		\Edge(x1'')(x3'')
		\Edge(x2'')(x4'');
		\end{tikzpicture}
	\end{center}
	\captionof{figure}{A graphic description of the spaces $F_r$ for $r=1,2,3$ with $a_r, b_r$ and the appropriate coefficients (i.e.\ $h^r$) attached to the graph's edges. Note that for $r=2$, $b_2$ depends on the continuation of the graph and might be $\infty$ (for example in the case where the graph is just continued by straight lines).}
	\label{ex_Fr}
\end{figure}
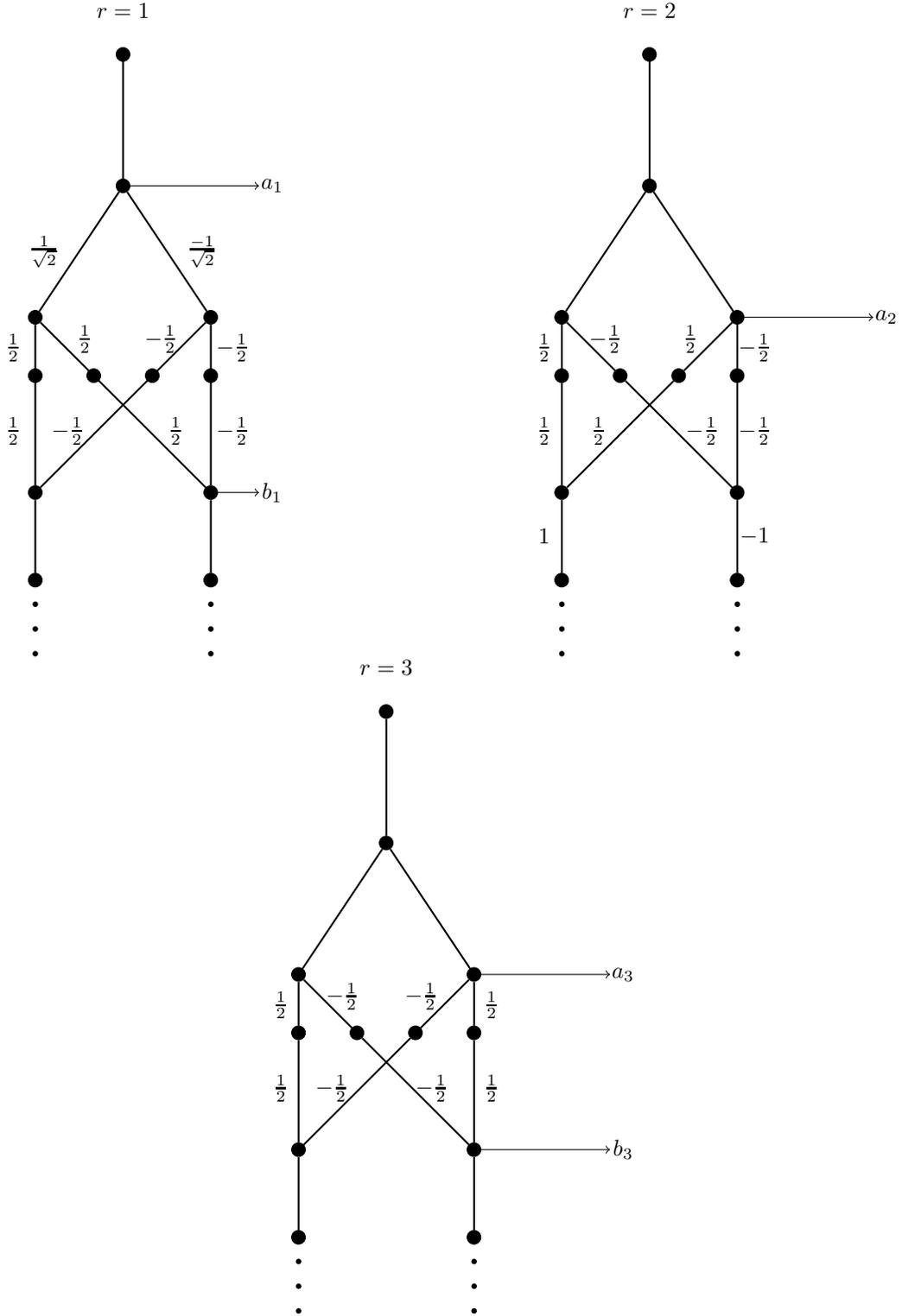
\subsection{Step 3 - Obtaining the one dimensional components}
The last step is to describe the one-dimensional operator to which $\Delta|_{F_r}$ is unitarily equivalent. The following computation shows that the scalars attached to the edges in Step 2 disappear under inner product. Let $\psi, \eta \in F_r$. Due to the fact that $\varphi_f$ is symmetric for every $f\in F_r$, for $t\geq 0$ we may define
\begin{equation} \label{eq:Ur}
U_r(f)(x)=\varphi_f(x)
\end{equation}
for some $x\in \Gamma$ with $|x|=t$. We have
\begin{align*}
\left\langle \psi,\eta \right\rangle_{L^2(\mu)} & =\int_{\Gamma}\varphi_{\psi}(x)h^{r}(x)\overline{\varphi_{\eta}(x)h^{r}(x)}d\mu(x)=\\
&
=\int_{0}^{h(\Gamma)}U_r(\psi)(t)\overline{U_r(\eta)(t)}\underset{=1}{\underbrace{\underset{|x|=t}{\sum}h^{r}(x)\overline{h^{r}(x)}}}dt=\\
& =\int_{0}^{h(\Gamma)}U_r(\psi)(t)\overline{U_r(\eta)(t)}dt=\\
& =\int_{a_{r}}^{b_{r}}U_r(\psi)(t)\overline{U_r(\eta)(t)}dt=\left\langle U_{r}(\psi),U_{r}(\eta)\right\rangle_{L^2([a_r,b_r])}.
\end{align*}
This implies that $U_r: F_r\rightarrow L^2([a_r,b_r])$ is a unitary operator. In addition, by the same reasoning as in the proof of Theorem \ref{thm_main}, $(ii)$, we have $U_{r}\Delta=\Delta_{r}U_{r}$, where $\Delta_{r}$ is the unbounded operator whose domain is $U_{r}(F_{r}\cap D(\Delta))$, and on that domain, $\Delta_{r}(f)=-f''$. We conclude that $\Delta|_{F_{r}}$ is unitarily equivalent to $\Delta_{r}$, which means that $\Delta\sim\underset{r\in\mathbb{N}}{\oplus}\Delta|_{F_{r}}\sim\underset{r\in\mathbb{N}}{\oplus}\Delta_{r}$. The rest of this subsection deals with analyzing the domains of the operators $\left(\Delta_r\right)_{r\in{\mathbb{N}}}$.

Divide the segment $(a_{r},b_{r})$ into the disjoint collection of segments
\[
\left\{(t_{j},t_{j+1})\right\}_{j\in A_{r}}\coloneqq\{I_{j}\}_{j\in A_{r}}.
\]

Note that for fixed $n$ and every $x\in\Gamma$, $|G_{x}\cap S_{n}|$ depends only on $|x|$. For $|x|=t$ ,let $w_{n}(t)=|G_x \cap S_n|$.
Now, by \eqref{eq:Ur}, and using \eqref{eq:hphi}, we may write explicitly
\begin{equation} \label{eq:UrForm}
U_r(f)(t)=\underset{|x|=t}{\sum}f(x)\frac{1}{\sqrt{g_{n(r)}(t)w_{n(r)}(t)}}\underset{u\in G_{x}\cap S_{n(r)}}{\sum}\phi_{0}^{r}(u).
\end{equation}
Note that for every $j\in A_{r}$, the functions $g_{n(r)}$ and $w_{n(r)}$ are constant on $I_{j}$ as they
change their values only at the $t_j$'s. Thus, $U_{r}(f)|_{I_{j}}=\varphi_{f}|_{I_{j}}$ is a linear combination of functions in $H^2(I_j)$,
so we have $U_r(f)|_{I_{j}}\in H^{2}(I_{j})$. In addition, we have
\[
\int_{a_r}^{b_r}|(U_{r}(f))''|^2dx=\int_{\Gamma}|\Delta f|^2d\mu<\infty.
\]

Functions in $D(\Delta)$ satisfy certain matching conditions on vertices,
which transform under $U_{r}$ into matching conditions on the connection
points of the segments $\left(I_{j}\right)_{j\in A_{r}}$. In order
to express these matching conditions, for every $j\in A_{r}$ we compute
\[
\varphi_{f}(t_{j}+)=\underset{\epsilon\rightarrow0}{lim}\,\varphi_{f}(t_{j}+\epsilon),\,\,\,\,\,\,\,\,\varphi_{f}(t_{j}-)=\underset{\epsilon\rightarrow0}{lim}\,\varphi_{f}(t_{j}-\epsilon)
\]
for $f\in F_{r}$. Let $\epsilon>0$.
\begin{align*}
\varphi_{f}(t_{j}+\epsilon) & =\left\langle f_{t_{j}+\epsilon},h_{t_{j}+\epsilon}^{r}\right\rangle=\\
& =\underset{|x|=t_{j}+\epsilon}{\sum}f(x)\frac{1}{\sqrt{g_{n}(t_{j}+\epsilon)w_{n}(t_{j}+\epsilon)}}\underset{u\in G_{x}\cap
	S_{n(r)}}{\sum}\phi_{0}^{r}(u)\\
& =\underset{v\in S_{j}}{\sum}\underset{|x|=t_{j}+\epsilon}{\underset{x\in
		G_{v}}{\sum}}f(x)\frac{1}{\sqrt{g_{n}(t_{j}+\epsilon)w_{n}(t_{j}+\epsilon)}}\underset{u\in G_{x}\cap S_{n(r)}}{\sum}\phi_{0}^{r}(u)
\end{align*}
The functions $g_{n},\,w_{n}$ are constant on $\left(t_{j},t_{j+1}\right)$,
so we may write $g_{n}(t_{j}+\epsilon)=g_{n}^{j}$ and $w_{n}(t_{j}+\epsilon)=w_{n}^{j}$
for sufficiently small $\epsilon$. Now, taking $\epsilon\rightarrow0$,
we have
\[
\varphi_{f}(t_{j}+)=\frac{1}{\sqrt{g_{n}^{j}w_{n}^{j}}}\underset{v\in S_{j}}{\sum}\underset{e:i(e)=v}{\sum}\underset{x\rightarrow
	v}{lim}\left((f|_{e})(x)\underset{\star}{\underbrace{\underset{u\in G_{x}\cap S_{n(r)}}{\sum}\phi_{0}^{r}(u)}}\right)
\]
Now, $\star$ is constant on every edge, so for every edge $e$ we
pick $x(e)\in e$ and we have
\begin{align*}
\varphi_{f}(t_{j}+) & =\frac{1}{\sqrt{g_{n(r)}^{j}w_{n(r)}^{j}}}\underset{v\in S_{j}}{\sum}\underset{e:i(e)=v}{\sum}\left(\underset{u\in
	G_{x(e)}\cap S_{n(r)}}{\sum}\phi_{0}^{r}(u)\right)f(v)=\\
& =\frac{1}{\sqrt{g_{n(r)}^{j}w_{n(r)}^{j}}}\underset{v\in S_{j}}{\sum}f(v)\underset{e:i(e)=v}{\sum}\left(\underset{u\in G_{x(e)}\cap
	S_{n(r)}}{\sum}\phi_{0}^{r}(u)\right)
\end{align*}
Now note that for every $v\in S_{j}$ we have that $\underset{e:i(e)=v}{\sum}\left(\underset{u\in G_{x(e)}\cap
	S_{n(r)}}{\sum}\phi_{0}^{r}(u)\right)$
is some multiple of $\underset{u\in G_{v}\cap S_{n(r)}}{\sum}\phi_{0}^{r}(u)$.
Thus, for $j=n(r)-1$, by the second part of Claim \ref{const_on_sj} we have that
$\varphi_{f}(t_{j}+)=0$. For $j\geq n(r)$, again by Claim \ref{const_on_sj} we
have that $\phi_{0}^{r}$ is constant on $G_{v}\cap S_{n(r)}$, and
for every $e$ such that $i(e)=v$ and $x\in e$, $G_{x}\cap S_{n(r)}=G_{v}\cap S_{n(r)}$.
Thus, $\phi_{0}^{r}|_{G_{x}\cap S_{n(r)}}\equiv\phi_{0}^{r}(u_{v})$
for some $u_{v}\in G_{v}\cap S_{n(r)}$. Now, if we denote (as before) by $b_{j}^{out}$
the number of edges emanating from vertices in $S_{j}$ (which is
the same for every vertex due to the symmetry of the graph), we have
\begin{align*}
\varphi_{f}(t_{j}+) & =\frac{1}{\sqrt{g_{n(r)}^{j}w_{n(r)}^{j}}}\underset{v\in S_{j}}{\sum}f(v)\cdot b_{j}^{out}\cdot
w_{n(r)}^{j}\phi_{0}^{r}(u_{v})=\\
& =\frac{w_{n(r)}^{j}b_{j}^{out}}{\sqrt{g_{n(r)}^{j}w_{n(r)}^{j}}}\underset{v\in S_{j}}{\sum}f(v)\phi_{0}^{r}(u_{v})=\\
& =\frac{\sqrt{w_{n(r)}^{j}}b_{j}^{out}}{\sqrt{g_{n(r)}^{j}}}\underset{v\in S_{j}}{\sum}f(v)\phi_{0}^{r}(u_{v})\\
\end{align*}
By a similar computation, it can be shown that $\varphi_{f}(b_{r})=0$
(in the case that $b_{r}<\infty$), and (with the proper notations)
that for $j\geq n(r)$
\[
\varphi_{f}(t_{j}-)=\frac{\sqrt{w_{n(r)}^{j-1}}b_{j}^{in}}{\sqrt{g_{n(r)}^{j-1}}}\underset{v\in S_{j}}{\sum}f(v)\phi_{0}^{r}(u_{v})
\]
Thus, for $j\geq n(r)$,
\[
U_{r}(f)(t_{j}+)=d_{j}^{r}U_{r}(f)(t_{j}-)
\]
where
\[
d_{j}^{r}=\frac{\sqrt{w_{n(r)}^{j}g_{n(r)}^{j-1}}b_{j}^{out}}{\sqrt{w_{n(r)}^{j-1}g_{n(r)}^{j}}b_{j}^{in}}.
\]
By a similar computation, using the matching condition of the derivative, we also have
\begin{align*}
U_r(f)'(t_j+)=c_{j}^{r}U_r(f)'(t_j-)
\end{align*}
where
\begin{align*}
c_{j}^{r}=\sqrt{\frac{w_{n(r)}^{j}g_{n(r)}^{j-1}}{w_{n(r)}^{j-1}g_{n(r)}^{j}}}.
\end{align*}
\begin{rem}\label{rem_dj}
	In the case that $r=0$, $n(r)=0$ as the function $\phi_{0}^{r}$ is chosen to be $\delta_o$. Thus $w_{n(r)}^{j}=1$ for every
	$j\in{\mathbb{N}}$. In addition, it can be seen that $b_{j}^{in}=\frac{g_{0}^{j-1}}{|S_j|}$ and $b_{j}^{out}=\frac{g_{0}^{j}}{|S_j|}$. From
	here, it can be seen that
	\begin{align*}
	d_{j}^{0}=\frac{\sqrt{g_{0}^{j}}}{\sqrt{g_{0}^{j-1}}}
	\end{align*}
	and that $c_j^0=\frac{1}{d_j^0}$.
\end{rem}
To conclude, $D(\Delta_{r})$ consists of all of the functions in
$L^{2}(a_{r},h(\Gamma))$ which satisfy the following conditions:
\begin{enumerate}
	\item $f(a_r)=0$. if $b_r<\infty$, then also $f(b_r)=0$.
	\item $\forall j\in A_{r}$ $f|_{I_{j}}\in H^{2}(I_{j})$.
	\item $\underset{j\in A_{r}}{\sum}\int_{I_{j}}|f^{''}(t)|dt<\infty$.
	\item $f(t_{j}+)=d_{j}^{r}f(t_{j}-)$ for every $n(r)\leq{j}\in A_{r}$.
	\item $f'(t_{j}+)=c_{j}^{r}f'(t_{j}-)$ for every $n(r)\leq j\in A_{r}$.
\end{enumerate}
In the case that $A_r$ is finite, functions in $D(\Delta_{r})$ also satisfy the condition $f(b_r)=0$.

For the example $\widetilde{\Gamma}$, the one dimensional version of the spaces $F_1,\,F_2$ from Figure \ref{ex_Fr} are shown in Figure \ref{ex_one_dim}. Note that functions in $D(\Delta_1)$ are defined on a compact segment (i.e. $b_1<\infty$) whereas functions in $D(\Delta_2)$ may not be compactly supported (for example, this is the case when the graph is continued by straight lines).

To conclude, the operator $\Delta_r$ is a Sturm-Liouville operator on a  $L^2(I)$ where $I\subseteq [0,\infty)$ with boundary conditions along points which correspond with $\left(t_l\right)_{l\in{A_r}}$. Since locally this operator acts as a second derivative, we henceforth refer to such
operators as `weighted Laplacians'.

\begin{figure}
	\begin{center}
		\begin{tikzpicture}[scale=1]
		\draw (0, 0.5) node {$r=1$};
		\vertex(v) at (0, 0) {};
		\vertex(w) at (0, -1.5) {};
		\draw (1.6, -1.5) node {$d_{2}^{1}=\sqrt{2},\, c_{2}^{2}=\frac{1}{\sqrt{2}}$};
		\vertex(u) at (0, -2.1666) {};
		\draw (1.6, -2.1666) node {$d_{3}^{1}=c_{3}^{1}=1$};
		\vertex(x1) at (0, -3.5) {};
		\Edge(v)(w)
		\Edge(w)(u)
		\Edge(u)(x1)
		\draw (6, 0.5) node {$r=2$};
		\vertex(v') at (6, 0) {};
		\vertex(w') at (6, -0.6666) {};
		\draw (7.6, -0.6666) node {$d_{3}^{2}=c_{3}^{2}=1$};
		\vertex(u') at (6, -2) {};
		\draw (7.6, -2) node {$d_{4}^{2}=\frac{1}{\sqrt{2}}\,c_{4}^{2}=\sqrt{2}$};
		\vertex(x1') at (6, -3) {};
		\Edge(v')(w')
		\Edge(w')(u')
		\Edge(u')(x1')
		\node[vertex, draw=none, fill=white](dots') at (6, -4.5) {};
		\path (x1') -- (dots') node [black, font=\Huge, midway, sloped] {$\dots$};
		\end{tikzpicture}
	\end{center}
	\captionof{figure}{The vertex conditions for the one dimensional versions of $F_1,\,F_2$.}
	\label{ex_one_dim}
\end{figure}
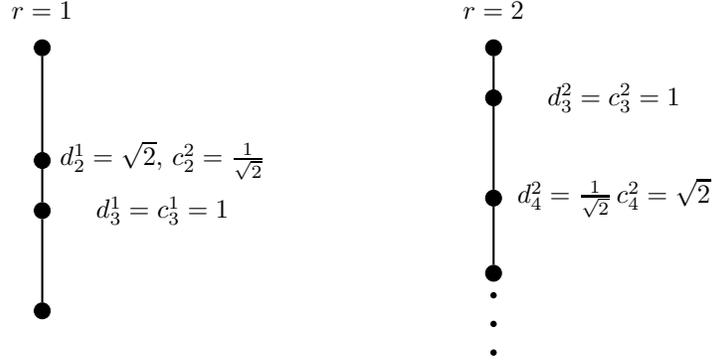

\subsection{Further Discussion}
In the algorithm presented in the proof of \cite[Theorem 2.6]{BrKe}, the spaces $\ell_2\left(S_n\right)$ are spanned iteratively where at each step, one spans the part of $\ell_2\left(S_n\right)$ which is not already spanned using a basis of mutual eigenfunctions of the collection $\Lambda_{n,+j}$. The functions obtained are the $\phi_0^r$, which we then use to construct $h^r$ above. As the characterization through mutual eigenfunctions of the family $\Lambda_{n,+j}$ is non-local, two natural questions come to mind regarding the description of the functions $\phi_0^r$ and $h^r$. 

The first question regards the choice of $\phi_{0}^{r}$ and asks whether there is some generic choice that can be made at each step and would work for any family preserving graph. In Section 5.1, we show that the decomposition described in \cite{NS} actually corresponds to a concrete choice of such a basis at each step, in the case where $\Gamma$ is a radial (spherically symmetric) tree. The choice there, however, is possible only for a tree, as it corresponds to cutting the tree into subtrees by deleting an edge between two generations. Another possible choice for $\phi_0^r$ comes to mind: perhaps it is possible to use roots of unity of a degree corresponding to the size of $S_{n(r)}$ in a universal way which would work for any family preserving graph. Figure \ref{ex_non_local} gives an example of two trees with $|S_2|=6$, where it is clear that the choice of $\phi_0$ supported on $S_2$ depends on the structure of $S_1$. By extension, it is clear that any choice of the functions $\phi_0^r$ on trees depends on the structure of the entire tree up to the relevant sphere. For general family preserving graphs, the choice would depend on the entire structure of the graph also beyond the relevant sphere, showing that a local universal choice is not possible.

Another question regards finding a more immediate connection between the edge weights, $h^r$, and the graph structure. In particular, since $h^r$ are constant on edges it is natural to consider them on the line graph and ask whether they can be obtained as eigenfunctions of some natural discrete operator on that graph, e.g.\ the adjacency matrix or the discrete Laplacian. While we cannot rule out the possibility of the existence of some operator with this property, Figure \ref{ex_not_ef} describes an example where even $h^0$ (the weight associated with the space of spherically symmetric functions on the graph) is not an eigenfunction of the adjacency matrix, $A$ or the Laplacian $\Delta_d$ on the line graph.

\begin{figure}
	\begin{center}
		\begin{tikzpicture}[scale=1]
		\vertex(v) at (0, 0) {};
		\vertex(w1) at (-1.5, -1.5) {};
		\vertex(w2) at (1.5, -1.5) {};
		\vertex(u1) at (-2.5, -3) {};
		\vertex(u2) at (-1.5, -3) {};
		\vertex(u3) at (-0.5, -3) {};
		\vertex(x1) at (0.5, -3) {};
		\vertex(x2) at (1.5, -3) {};
		\vertex(x3) at (2.5, -3) {};W
		\node[vertex, draw=none, fill=white](dots1) at (-2.5, -4) {};
		\node[vertex, draw=none, fill=white](dots2) at (-1.5, -4) {};
		\node[vertex, draw=none, fill=white](dots3) at (-0.5, -4) {};
		\node[vertex, draw=none, fill=white](dots4) at (0.5, -4) {};
		\node[vertex, draw=none, fill=white](dots5) at (1.5, -4) {};
		\node[vertex, draw=none, fill=white](dots6) at (2.5, -4) {};
		\path (u1) -- (dots1) node [black, font=\Huge, midway, sloped] {$\dots$};
		\path (u2) -- (dots2) node [black, font=\Huge, midway, sloped] {$\dots$};
		\path (u3) -- (dots3) node [black, font=\Huge, midway, sloped] {$\dots$};
		\path (x1) -- (dots4) node [black, font=\Huge, midway, sloped] {$\dots$};
		\path (x2) -- (dots5) node [black, font=\Huge, midway, sloped] {$\dots$};
		\path (x3) -- (dots6) node [black, font=\Huge, midway, sloped] {$\dots$};
		\Edge(v)(w1)
		\Edge(v)(w2)
		\Edge(w1)(u1)
		\Edge(w1)(u2)
		\Edge(w1)(u3)
		\Edge(w2)(x1)
		\Edge(w2)(x2)
		\Edge(w2)(x3);
		\vertex(v') at (6, 0) {};
		\vertex(w1') at (4.5, -1.5) {};
		\vertex(w2') at (6, -1.5) {};
		\vertex(w3') at (7.5, -1.5) {};
		\vertex(u1') at (4, -3) {};
		\vertex(u2') at (5, -3) {};
		\vertex(x1') at (5.5, -3) {};
		\vertex(x2') at (6.5, -3) {};
		\vertex(y1') at (7, -3) {};
		\vertex(y2') at (8, -3) {};
		\node[vertex, draw=none, fill=white](dots1') at (4, -4) {};
		\node[vertex, draw=none, fill=white](dots2') at (5, -4) {};
		\node[vertex, draw=none, fill=white](dots3') at (5.5, -4) {};
		\node[vertex, draw=none, fill=white](dots4') at (6.5, -4) {};
		\node[vertex, draw=none, fill=white](dots5') at (7, -4) {};
		\node[vertex, draw=none, fill=white](dots6') at (8, -4) {};
		\path (u1') -- (dots1') node [black, font=\Huge, midway, sloped] {$\dots$};
		\path (u2') -- (dots2') node [black, font=\Huge, midway, sloped] {$\dots$};
		\path (x1') -- (dots3') node [black, font=\Huge, midway, sloped] {$\dots$};
		\path (x2') -- (dots4') node [black, font=\Huge, midway, sloped] {$\dots$};
		\path (y1') -- (dots5') node [black, font=\Huge, midway, sloped] {$\dots$};
		\path (y2') -- (dots6') node [black, font=\Huge, midway, sloped] {$\dots$};
		\Edge(v')(w1')
		\Edge(v')(w2')
		\Edge(v')(w3')
		\Edge(w1')(u1')
		\Edge(w1')(u2')
		\Edge(w2')(x1')
		\Edge(w2')(x2')
		\Edge(w3')(y1')
		\Edge(w3')(y2');
		\end{tikzpicture}
	\end{center}
	\captionof{figure}{In this example, the number of vertices in $S_2$ is the same in both graphs ($6$). Using the fact that the elements of the collection $\phi_{0}^{r}$ which are supported on $S_2$ are orthogonal to $\ell_2\left(S_1\right)$, one can easily show that it is not possible for the spanning collections on both graphs to be the same.}
	\label{ex_non_local}
\end{figure}
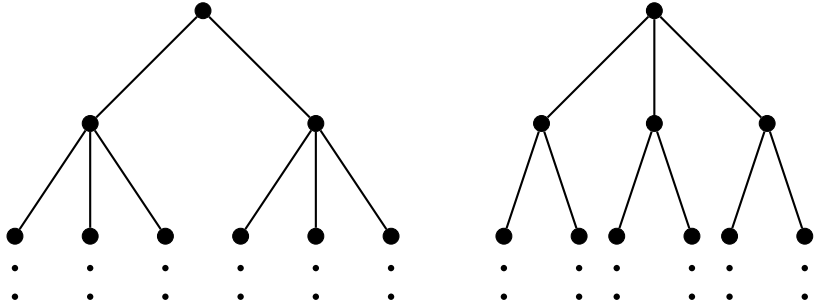

\begin{figure}
	\begin{center}
		\begin{tikzpicture}[scale=1]
		\vertex(o) at (0,0) {};
		\vertex(v) at (0,-1.5) {};
		\vertex(w1) at (-1,-3) {};
		\vertex(w2) at (0,-3) {};
		\vertex(w3) at (1,-3) {};
		\vertex(u) at (0,-4.5) {};
		\vertex(x1) at (-1,-6) {};
		\vertex(x2) at (0,-6) {};
		\vertex(x3) at (1,-6) {};
		\vertex(z) at (0,-7.5) {};
		\vertex(y1) at (-1,-9) {};
		\vertex(y2) at (1,-9) {};
		\Edge(o)(v)
		\Edge(v)(w1)
		\Edge(v)(w2)
		\Edge(v)(w3)
		\Edge(w1)(u)
		\Edge(w2)(u)
		\Edge(w3)(u)
		\Edge(u)(x1)
		\Edge(u)(x2)
		\Edge(u)(x3)
		\Edge(x1)(z)
		\Edge(x2)(z)
		\Edge(x3)(z)
		\Edge(z)(y1)
		\Edge(z)(y2);
		\node[vertex, draw=none, fill=white](dots1) at (-1, -10) {};
		\path (y1) -- (dots1) node [black, font=\Huge, midway, sloped] {$\dots$};
		\node[vertex, draw=none, fill=white](dots2) at (1, -10) {};
		\path (y2) -- (dots2) node [black, font=\Huge, midway, sloped] {$\dots$};
		\vertex(o') at (5,0) {};
		\vertex(u1') at (4,-1.5) {};
		\vertex(u2') at (5,-1.5) {};
		\vertex(u3') at (6,-1.5) {};
		\vertex(w1') at (4,-3) {};
		\vertex(w2') at (5,-3) {};
		\vertex(w3') at (6,-3) {};
		\vertex(v1') at (4,-4.5) {};
		\vertex(v2') at (5,-4.5) {};
		\vertex(v3') at (6,-4.5) {};
		\vertex(x1') at (4,-6) {};
		\vertex(x2') at (5,-6) {};
		\vertex(x3') at (6,-6) {};
		\vertex(y1') at (4,-7.5) {};
		\vertex(y2') at (6,-7.5) {};
		\Edge(o')(u1')
		\Edge(o')(u2')
		\Edge(o')(u3')
		\Edge(v1')(w1')
		\Edge(v1')(w2')
		\Edge(v1')(w3')
		\Edge(v2')(w1')
		\Edge(v2')(w2')
		\Edge(v2')(w3')
		\Edge(v3')(w1')
		\Edge(v3')(w2')
		\Edge(v3')(w3')
		\Edge(w1')(u1')
		\Edge(w2')(u2')
		\Edge(w3')(u3')
		\Edge(v1')(x1')
		\Edge(v2')(x2')
		\Edge(v3')(x3')
		\Edge(x1')(y1')
		\Edge(x1')(y2')
		\Edge(x2')(y1')
		\Edge(x2')(y2')
		\Edge(x3')(y1')
		\Edge(x3')(y2');
		\node[vertex, draw=none, fill=white](dots1') at (4, -8.5) {};
		\path (y1') -- (dots1') node [black, font=\Huge, midway, sloped] {$\dots$};
		\node[vertex, draw=none, fill=white](dots2') at (6, -8.5) {};
		\path (y2') -- (dots2') node [black, font=\Huge, midway, sloped] {$\dots$};
		\end{tikzpicture}
	\end{center}
	\captionof{figure}{The weights associated with the symmetric component of the graph on the left side are symmetric themselves. Thus, for every $k$, the weight of edges on the $k$'th sphere is just $\frac{1}{\sqrt{n_k}}$ where $n_k$ is the number of edges on that sphere. In this case, one can easily verify that the edge weights are not eigenfunctions of the adjancency operator nor of the Laplacian on the line graph (on the right side).}
	\label{ex_not_ef}
\end{figure}
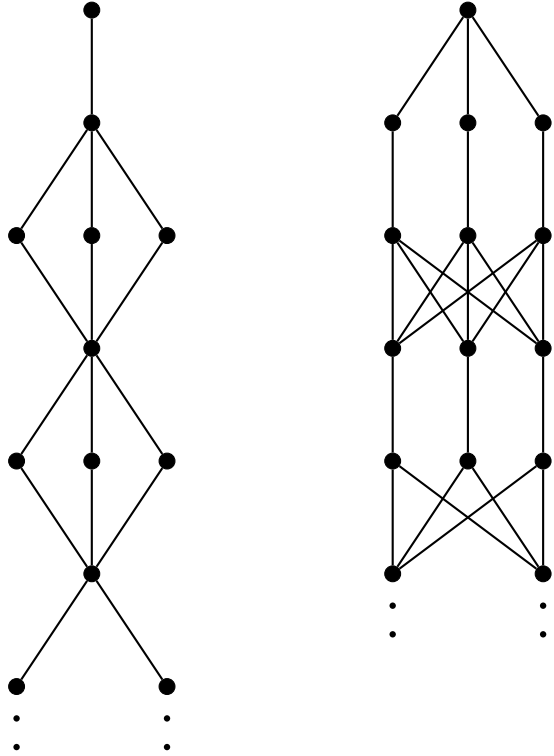


\section{Examples}

In this section we first describe how the results of \cite{NS, Sol} are a special case of our analysis here. Then, in order to demonstrate the
utility of our method, we apply it to obtain a result on metric antitrees.

\subsection{Radial Trees}

A metric tree $\Gamma$ is called \emph{radial (or regular) }if
for every $x,y\in S_{n}$, $deg(x)=deg(y)$, and for any two edges of the same generation $e_{1},\,e_{2}$, $l(e_{1})=l(e_{2})$.
\begin{claim}
Every regular tree is family preserving.
\end{claim}

\begin{proof}
Let $\Gamma$ be a regular tree, and let $v,u\in S_{n}$ be backward
neighbors. $\Gamma$ is a tree, so the sets $A_{v}=\left\{w\in V(\Gamma)\,:\,v\preceq w\right\}$,
$A_{u}=\left\{w\in V(\Gamma)\,:\,u\preceq w\right\}$ are disjoint. Furthermore,
regularity of $\Gamma$ implies that the graphs induced by those sets
are isomorphic as rooted graphs, with $u,v$ as roots. Let $\sigma:\,A_{v}\rightarrow A_{u}$
be an isomorphism between the induced graphs. Define
\[
\tau(w)=\begin{cases}
\sigma(w) & w\in A_{v}\cup A_{u}\\
w & else
\end{cases}.
\]
It is easy to verify that $\tau$ is a rooted graph automorphism which
satisfies $\tau|_{S_{n-j}}=Id$ for every $j\geq1$. $\Gamma$ is
a tree, so there are no forward neighbors, and that finishes the proof.
\end{proof}
A decomposition for the Laplacian on regular trees was presented in
\cite{NS,Sol}. We will present the decomposition here, and
show that it is a special case of the decomposition presented in the
proof of Theorem \ref{thm_main}.

\vbox{For every $v\in{V(\Gamma)}$, denote by $b(v)$ the number of edges emanating from $v$ (i.e., the edges going `away' from the root), and order the edges $\left(e_{v}^{1},\ldots,e_{v}^{b(v)}\right)$. For every $u\in{V(\Gamma)}$, let $A_{u}=\{w\in V(\Gamma)\,:\,u\preceq w\}$. For every $1\leq{j}\leq{b(v)}$, let $G_{e_{v}^{j}}$ be the subtree of $\Gamma$ generated by $A_{t\left(e_{v}^{j}\right)}\cup \left\{v\right\}$ (see Figure \ref{ex_subtree_def} for an illustration of these definitions).}
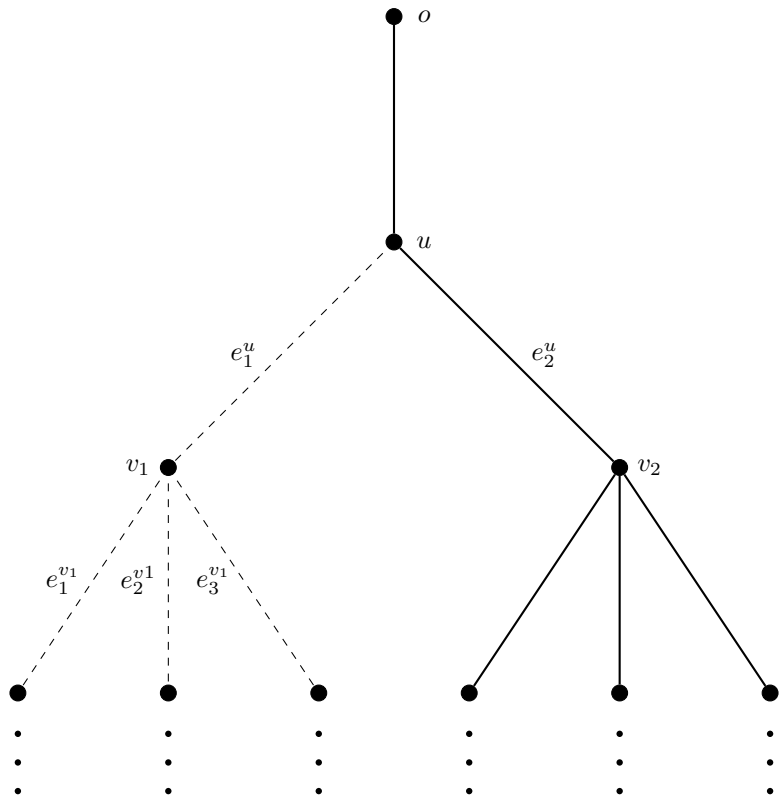
\begin{figure}
	\begin{center}
		\begin{tikzpicture}[scale=2]
		\vertex(o) at (0,0) {};
		\draw(0.2,0) node{$o$};
		\vertex(u) at (0,-1.5) {};
		\draw(0.2,-1.5) node{$u$};
		\vertex(v1) at (-1.5,-3) {};
		\draw(-1.7,-3) node{$v_1$};
		\vertex(v2) at (1.5,-3) {};
		\draw(1.7,-3) node{$v_2$};
		\vertex(w1) at (-2.5,-4.5) {};
		\vertex(w2) at (-1.5,-4.5) {};
		\vertex(w3) at (-0.5,-4.5) {};
		\vertex(w4) at (0.5,-4.5) {};
		\vertex(w5) at (1.5, -4.5) {};
		\vertex(w6) at (2.5, -4.5) {};
		\node[vertex, draw=none, fill=white](dots1) at (-2.5, -5.5) {};
		\node[vertex, draw=none, fill=white](dots2) at (-1.5, -5.5) {};
		\node[vertex, draw=none, fill=white](dots3) at (-0.5, -5.5) {};
		\node[vertex, draw=none, fill=white](dots4) at (0.5, -5.5) {};
		\node[vertex, draw=none, fill=white](dots5) at (1.5, -5.5) {};
		\node[vertex, draw=none, fill=white](dots6) at (2.5, -5.5) {};
		\path (w1) -- (dots1) node [black, font=\Huge, midway, sloped] {$\dots$};
		\path (w2) -- (dots2) node [black, font=\Huge, midway, sloped] {$\dots$};
		\path (w3) -- (dots3) node [black, font=\Huge, midway, sloped] {$\dots$};
		\path (w4) -- (dots4) node [black, font=\Huge, midway, sloped] {$\dots$};
		\path (w5) -- (dots5) node [black, font=\Huge, midway, sloped] {$\dots$};
		\path (w6) -- (dots6) node [black, font=\Huge, midway, sloped] {$\dots$};
		\Edge(o)(u)
		\draw (0,-1.5)--(-1.5,-3)[dashed];
		\draw(-1,-2.25) node{$e_{1}^{u}$};
		\Edge(u)(v2)
		\draw (1,-2.25) node{$e_{2}^{u}$};
		\draw (-1.5,-3)--(-2.5,-4.5)[dashed];
		\draw (-2.2,-3.75) node{$e_{1}^{v_{1}}$};
		\draw (-1.5,-3)--(-1.5,-4.5)[dashed];
		\draw (-1.7, -3.75) node{$e_{2}^{v{1}}$};
		\draw (-1.5,-3)--(-0.5,-4.5)[dashed];
		\draw (-1.2,-3.75) node{$e_{3}^{v_{1}}$};
		\Edge(v2)(w4)
		\Edge(v2)(w5)
		\Edge(v2)(w6);
		\end{tikzpicture}
	\end{center}
	\captionof{figure}{The dashed edges stand for $G_{e_{1}^{u}}$, and the dots signify some radial continuation of the tree.}
	\label{ex_subtree_def}
\end{figure}
Let $\omega$ be the root of unity of order $b(v)$ (i.e.\ $\omega=e^{\frac{2\pi i}{b(v)}}$).
Let $s\in\mathbb{N},\,1\leq s\leq b(v)-1$. Let $M_{v}^{\langle s\rangle}$
be the space of all functions $f\in L^{2}(\Gamma)$ such
 that $supp(f)\subseteq{\underset{k=1}{\overset{b(v)}{\bigcup}}G_{e_{v}^{k}}}$ and there
exists a measurable function $g:\,(t_{gen(v)},\infty)\rightarrow\mathbb{C}$
for which $f(x)=g(|x|)\omega^{js}\iff x\in G_{e_{v}^{j}},\text{ } 1\leq{j}\leq{b(v)}$. Also, denote by $M_{\Gamma}$ the space of all functions in $L^2(\Gamma)$ that are radially symmetric.

The Naimark-Solomyak decomposition is described in the following theorem from \cite{Sol}.

\begin{thm}\emph{(\cite[Theorem 3.2]{Sol})} \label{thm:NS}
Let $\Gamma$ be a regular (radially symmetric and spherically homogeneous) metric tree such that $b(o)=1$. Then the subspaces $M_{v}^{\langle s\rangle},\, v\in{V(\Gamma)},\, b(v)>1,\, 1\leq{s}\leq{b(v)-1}$ are mutually orthogonal and orthogonal to $M_{\Gamma}$. Moreover,
$$L^2(\Gamma)=M_{\Gamma}\oplus\left(\underset{\substack{v\in{V(\Gamma)} \\ b(v)>1 \\ 1\leq{s}\leq{b(v)-1}}}{\oplus}M_{v}^{\langle s\rangle}\right)$$
and this decomposition reduces the Dirichlet-Kirchhoff Laplacian on $\Gamma$.
\end{thm}

It is apparent from their description that for every $v\in{V(\Gamma)\setminus{o}}$ and every $1\leq s\leq b(v)$, the space $M_{v}^{\langle s\rangle}$ can be described by attaching scalar coefficients to the tree's edges (see Figure \ref{ex_mvs_weights}) in the same manner shown in Section 4 (Figure \ref{ex_Fr}).
\begin{figure}
	\begin{center}
		\begin{tikzpicture}[scale=2]
		\vertex(o) at (0,0) {};
		\vertex(u) at (0,-1) {};
		\vertex(v1) at (-1.5,-2) {};
		\draw (-1.5,-1.85) node{$v$};
		\vertex(v2) at (1.5,-2) {};
		\vertex(w1) at (-2.5,-3) {};
		\vertex(w2) at (-1.5,-3) {};
		\vertex(w3) at (-0.5,-3) {};
		\vertex(w4) at (0.5,-3) {};
		\vertex(w5) at (1.5,-3) {};
		\vertex(w6) at (2.5,-3) {};
		\vertex(x1) at (-2.8,-4) {};
		\node[vertex, draw=none, fill=white](dots1) at (-2.8, -4.5) {};
		\path (x1) -- (dots1) node [black, font=\Huge, midway, sloped] {$\dots$};
		\vertex(x2) at  (-2.2,-4) {};
		\node[vertex, draw=none, fill=white](dots2) at (-2.2, -4.5) {};
		\path (x2) -- (dots2) node [black, font=\Huge, midway, sloped] {$\dots$};
		\vertex(x3) at (-1.8,-4) {};
		\node[vertex, draw=none, fill=white](dots3) at (-1.8, -4.5) {};
		\path (x3) -- (dots3) node [black, font=\Huge, midway, sloped] {$\dots$};
		\vertex(x4) at (-1.2,-4) {};
		\node[vertex, draw=none, fill=white](dots4) at (-1.2, -4.5) {};
		\path (x4) -- (dots4) node [black, font=\Huge, midway, sloped] {$\dots$};
		\vertex(x5) at (-0.8,-4) {};
		\node[vertex, draw=none, fill=white](dots5) at (-0.8, -4.5) {};
		\path (x5) -- (dots5) node [black, font=\Huge, midway, sloped] {$\dots$};
		\vertex(x6) at (-0.2,-4) {};
		\node[vertex, draw=none, fill=white](dots6) at (-0.2, -4.5) {};
		\path (x6) -- (dots6) node [black, font=\Huge, midway, sloped] {$\dots$};
		\vertex(x7) at (0.2,-4) {};
		\node[vertex, draw=none, fill=white](dots7) at (0.2, -4.5) {};
		\path (x7) -- (dots7) node [black, font=\Huge, midway, sloped] {$\dots$};
		\vertex(x8) at (0.8,-4) {};
		\node[vertex, draw=none, fill=white](dots8) at (0.8, -4.5) {};
		\path (x8) -- (dots8) node [black, font=\Huge, midway, sloped] {$\dots$};
		\vertex(x9) at (1.2,-4) {};
		\node[vertex, draw=none, fill=white](dots9) at (1.2, -4.5) {};
		\path (x9) -- (dots9) node [black, font=\Huge, midway, sloped] {$\dots$};
		\vertex(x10) at (1.8,-4) {};
		\node[vertex, draw=none, fill=white](dots10) at (1.8, -4.5) {};
		\path (x10) -- (dots10) node [black, font=\Huge, midway, sloped] {$\dots$};
		\vertex(x11) at (2.2,-4) {};
		\node[vertex, draw=none, fill=white](dots11) at (2.2, -4.5) {};
		\path (x11) -- (dots11) node [black, font=\Huge, midway, sloped] {$\dots$};
		\vertex(x12) at (2.8,-4) {};
		\node[vertex, draw=none, fill=white](dots12) at (2.8, -4.5) {};
		\path (x12) -- (dots12) node [black, font=\Huge, midway, sloped] {$\dots$};
		\Edge(o)(u)
		\Edge(u)(v1)
		\Edge(u)(v2)
		\Edge(v1)(w1)
		\draw (-2.2,-2.5) node{$\frac{\omega}{\sqrt{3}}$};
		\Edge(v1)(w2)
		\draw (-1.65,-2.5) node{$\frac{\omega^2}{\sqrt{3}}$};
		\Edge(v1)(w3)
		\draw (-0.8,-2.5) node{$\frac{\omega^3}{\sqrt{3}}$};
		\Edge(v2)(w4)
		\Edge(v2)(w5)
		\Edge(v2)(w6)
		\Edge(w1)(x1)
		\draw (-2.8,-3.5) node{$\frac{\omega}{\sqrt{6}}$};
		\Edge(w1)(x2)
		\draw (-2.2,-3.5) node{$\frac{\omega}{\sqrt{6}}$};
		\Edge(w2)(x3)
		\draw (-1.8,-3.5) node{$\frac{\omega^2}{\sqrt{6}}$};
		\Edge(w2)(x4)
		\draw (-1.2,-3.5) node{$\frac{\omega^2}{\sqrt{6}}$};
		\Edge(w3)(x5)
		\draw (-0.8,-3.5) node{$\frac{\omega^3}{\sqrt{6}}$};
		\Edge(w3)(x6)
		\draw (-0.2,-3.5) node{$\frac{\omega^3}{\sqrt{6}}$};
		\Edge(w4)(x7)
		\Edge(w4)(x8)
		\Edge(w5)(x9)
		\Edge(w5)(x10)
		\Edge(w6)(x11)
		\Edge(w6)(x12);
		\vertex(o') at (0,-6) {};
		\vertex(u') at (0,-7) {};
		\vertex(v1') at (-1.5,-8) {};
		\draw (-1.5,-7.85) node{$v$};
		\vertex(v2') at (1.5,-8) {};
		\vertex(w1') at (-2.5,-9) {};
		\vertex(w2') at (-1.5,-9) {};
		\vertex(w3') at (-0.5,-9) {};
		\vertex(w4') at (0.5,-9) {};
		\vertex(w5') at (1.5,-9) {};
		\vertex(w6') at (2.5,-9) {};
		\vertex(x1') at (-2.8,-10) {};
		\node[vertex, draw=none, fill=white](dots1') at (-2.8,-10.5) {};
		\path (x1') -- (dots1') node [black, font=\Huge, midway, sloped] {$\dots$};
		\vertex(x2') at  (-2.2,-10) {};
		\node[vertex, draw=none, fill=white](dots2') at (-2.2,-10.5) {};
		\path (x2') -- (dots2') node [black, font=\Huge, midway, sloped] {$\dots$};
		\vertex(x3') at (-1.8,-10) {};
		\node[vertex, draw=none, fill=white](dots3') at (-1.8,-10.5) {};
		\path (x3') -- (dots3') node [black, font=\Huge, midway, sloped] {$\dots$};
		\vertex(x4') at (-1.2,-10) {};
		\node[vertex, draw=none, fill=white](dots4') at (-1.2,-10.5) {};
		\path (x4') -- (dots4') node [black, font=\Huge, midway, sloped] {$\dots$};
		\vertex(x5') at (-0.8,-10) {};
		\node[vertex, draw=none, fill=white](dots5') at (-0.8,-10.5) {};
		\path (x5') -- (dots5') node [black, font=\Huge, midway, sloped] {$\dots$};
		\vertex(x6') at (-0.2,-10) {};
		\node[vertex, draw=none, fill=white](dots6') at (-0.2,-10.5) {};
		\path (x6') -- (dots6') node [black, font=\Huge, midway, sloped] {$\dots$};
		\vertex(x7') at (0.2,-10) {};
		\node[vertex, draw=none, fill=white](dots7') at (0.2,-10.5) {};
		\path (x7') -- (dots7') node [black, font=\Huge, midway, sloped] {$\dots$};
		\vertex(x8') at (0.8,-10) {};
		\node[vertex, draw=none, fill=white](dots8') at (0.8,-10.5) {};
		\path (x8') -- (dots8') node [black, font=\Huge, midway, sloped] {$\dots$};
		\vertex(x9') at (1.2,-10) {};
		\node[vertex, draw=none, fill=white](dots9') at (1.2,-10.5) {};
		\path (x9') -- (dots9') node [black, font=\Huge, midway, sloped] {$\dots$};
		\vertex(x10') at (1.8,-10) {};
		\node[vertex, draw=none, fill=white](dots10') at (1.8,-10.5) {};
		\path (x10') -- (dots10') node [black, font=\Huge, midway, sloped] {$\dots$};
		\vertex(x11') at (2.2,-10) {};
		\node[vertex, draw=none, fill=white](dots11') at (2.2,-10.5) {};
		\path (x11') -- (dots11') node [black, font=\Huge, midway, sloped] {$\dots$};
		\vertex(x12') at (2.8,-10) {};
		\node[vertex, draw=none, fill=white](dots12') at (2.8,-10.5) {};
		\path (x12') -- (dots12') node [black, font=\Huge, midway, sloped] {$\dots$};
		\Edge(o')(u')
		\Edge(u')(v1')
		\Edge(u')(v2')
		\Edge(v1')(w1')
		\draw (-2.2,-8.5) node{$\frac{\omega^2}{\sqrt{3}}$};
		\Edge(v1')(w2')
		\draw (-1.65,-8.5) node{$\frac{\omega^4}{\sqrt{3}}$};
		\Edge(v1')(w3')
		\draw (-0.8,-8.5) node{$\frac{\omega^6}{\sqrt{3}}$};
		\Edge(v2')(w4')
		\Edge(v2')(w5')
		\Edge(v2')(w6')
		\Edge(w1')(x1')
		\draw (-2.8,-9.5) node{$\frac{\omega^2}{\sqrt{6}}$};
		\Edge(w1')(x2')
		\draw (-2.2,-9.5) node{$\frac{\omega^2}{\sqrt{6}}$};
		\Edge(w2')(x3')
		\draw (-1.8,-9.5) node{$\frac{\omega^4}{\sqrt{6}}$};
		\Edge(w2')(x4')
		\draw (-1.2,-9.5) node{$\frac{\omega^4}{\sqrt{6}}$};
		\Edge(w3')(x5')
		\draw (-0.8,-9.5) node{$\frac{\omega^6}{\sqrt{6}}$};
		\Edge(w3')(x6')
		\draw (-0.2,-9.5) node{$\frac{\omega^6}{\sqrt{6}}$};
		\Edge(w4')(x7')
		\Edge(w4')(x8')
		\Edge(w5')(x9')
		\Edge(w5')(x10')
		\Edge(w6')(x11')
		\Edge(w6')(x12');
		\end{tikzpicture}
	\end{center}
	\captionof{figure}{An example of the scalar coefficients defining the spaces $M_{v}^{s}$ for $s=1,2$ (the upper graph is for $s=1$ and the lower one is for $s=2$). Functions in these spaces are supported only on the subtree whose root is $v$.}
	\label{ex_mvs_weights}
\end{figure}
In order to show that the above decomposition is a special case of the decomposition we describe here, we first associate with each space $M_{v}^{\langle s\rangle}$ a function defined on the associated discrete graph and supported on a sphere.

\begin{lem}
For every $v\in V(\Gamma)$, and every $s=1,\ldots,b(v)-1$, the space
$M_{v}^{\langle s\rangle}$ is equal to the spaces $F_{\phi_{v}^{s}}$,
where $\phi_{v}^{s}$ is defined in the following way:
\[
\phi_{v}^{s}(u)=\begin{cases}
\frac{\omega^{js}}{\sqrt{b(v)}} & e_{v}^{j}=(v,u)\\
0 & else
\end{cases}
\]
where $\omega=e^{\frac{2\pi i}{b(v)}}$.
\end{lem}

\begin{proof}
We first show that $M_{v}^{\langle s\rangle}\subseteq{F_{\phi_{v}^{s}}}$. Let $f\in M_{v}^{\langle s\rangle}$ and let $g:(t_{gen(v)},\infty)\rightarrow{\mathbb{C}}$ such that $f(x)=g(|x|)\omega^{js}\iff{x\in{G_{e_{j}^{v}}}}$. Also denote $k=gen(v)+1$. It is sufficient to show that $P_{\phi_{v}^{s}}(f)=f$. Let $x\in{\Gamma}$. Assume first that $x\in{G_{e_{v}^{j}}}$ for some $1\leq{j}\leq{b(v)}$ and that $x\neq{v}$. The fact that $\Gamma$ is a tree implies that $G_x\cap{S_{gen(v)+1}}=\{u\}$ where $x\in{G_{(v,u)}}$. Thus, $h^{\phi_{v}^{s}}(x)=\frac{\omega^{js}}{\sqrt{g_k(|x|)b(v)}}$. A direct computation shows that $\left\langle{f_{|x|}},{h^{\phi_{v}^{s}}_{|x|}}\right\rangle=g(|x|)\sqrt{g_k(|x|)b(v)}$ which means that $P_{\phi_{v}^{s}}(f)(x)=g(|x|)\omega^{js}=f(x)$. for $x\notin{G_{e_{v}^{j}}}$ for every $1\leq{j}\leq{b(v)}$, $h^{\phi_{v}^{s}}(x)=0$ which means that $P_{\phi_{v}^{s}}(f)(x)=0$.

The inclusion ${F_{\phi_{v}^{s}}}\subseteq M_{v}^{\langle s\rangle}$ easily follows by defining $g(|x|)=\frac{1}{\sqrt{g_k(|x|)b(v)}}\left\langle{f_{|x|}},{h^{\phi_{v}^{s}}_{|x|}}\right\rangle$ for $f \in{F_{\phi_{v}^{s}}}$.
\end{proof}

We conclude our demonstration by showing that the functions $\left\{ \phi_v^s \right \}$ satisfy the properties of $\left\{\phi_{0}^{r} \right\}$ in Theorem \ref{thm_br}. Thus the decomposition described in Theorem \ref{thm:NS} is indeed a particular case of the decomposition described in Theorem \ref{thm_main}.

For every $v\in{V(\Gamma)}$ such that $b(v)>1$ and $1\leq{s}\leq{b(v)-1}$, denote $H_{v}^{s}\coloneqq{\overline{sp\left\{\phi_{v}^{s},\Delta_{d}(\phi_{v}^{s}),\ldots\right\}}}$, $H_0\coloneqq{\overline{sp\left\{\delta_0,\Delta_{d}(\delta_0),\ldots\right\}}}$. Also, denote $V'=\left\{v\in{V(\Gamma)\setminus\{o\}} : b(v)>1\right\}$.
\begin{thm} \label{thm_reduct}
	Let $\Gamma$ be a radially symmetric tree. Then $\ell^2(V(\Gamma))=H_0\oplus\left(\underset{v\in{V'}}{\oplus} H_{v}^{s}\right)$. Furthermore, for every $v\in{V'}$ and $1\leq{s}\leq{b(v)-1}$,
	
	$(i)$ There exists $n(v,s)\in{\mathbb{N}}$ such that $supp\left(\phi_{v}^{s}\right)\subseteq{S_{n(v,s)}}$.
	
	$(ii)$ The set $\left\{\phi_{0}^{v,s},\phi_{1}^{v,s},\ldots\right\}$ obtained from $\left\{\phi_{v}^{s},\Delta_{d}(\phi_{v}^{s}),\ldots\right\}$ by applying the Gram-Schmidt process has the property that $supp\left(\phi_{k}^{v,s}\right)\subseteq{S_{n(v,s)+k}}$.
	
	$(iii)$ For every $j\in{\mathbb{N}}$, $\phi_{0}^{v,s}$ is an eigenvector of $\Lambda_{n(v,s),\pm{j}}$.
\end{thm}
\begin{proof}
	Part $(i)$ follows directly from the definition of $\phi_{v}^{s}$. Parts $(ii)$ and $(iii)$ also follow from the definition of $\phi_{v}^{s}$, using in addition the symmetry of the graph. The mutual orthogonality of the family of spaces $H_0$, $H_{v}^{s}$ $\left(v \in V'\right)$ is an elementary but somewhat tedious computation and the fact that $\ell^2(V(\Gamma))=H_0\oplus\left(\underset{v\in{V'}}{\oplus} H_{v}^{s}\right)$ is a simple local dimension counting.
\end{proof}

\begin{rem}
The procedure of obtaining the functions $\phi_0^r$ of Theorem \ref{thm_br} is a constructive procedure which proceeds by constructing cyclic subspaces recursively. One starts with the cyclic space spanned by the delta function at the root and at each step one chooses a function that is orthogonal to all the previously constructed cyclic subspaces and supported on a sphere not yet `covered' by the procedure. Thus, there often is some freedom in the choice of the functions $\phi_0^r$. Remarkably, though, due to the symmetry of family preserving graphs, as long as these functions are chosen as eigenfunctions of the appropriate operators (the $\Lambda_{n,\pm{j}}$) the associated Jacobi matrices do not depend on the choice of these functions.

Similarly, in the associated metric case covered by Theorem \ref{thm_main}, the one-dimensional operators obtained in the decomposition are independent of the choice of the $\phi_0^r$ described in Theorem \ref{thm_br}. This can be seen from the fact that these one-dimensional operators as described in Section 4 above, do not depend on the choice of $\phi_0^r$.

Thus, while the discussion above shows that the method of decomposition of \cite{NS, Sol} can be realized by a particular choice of $\phi_0^r$ we actually obtain the somewhat more interesting result that the decomposition itself is independent of the particular choice of these functions (among the possibilities allowed by Theorem \ref{thm_br}).
\end{rem}


\subsection{Antitrees}

An antitree is a rooted graph, $G$, with all possible edges between any two neighboring spheres (and no other edges). Formally, a rooted graph
is an antitree if $\forall n\in\mathbb{N},\ \forall u\in S_{n}, \ \forall v\in S_{n+1},\,(u,v)\in E$.
In a sense, an antitree is a bipartite antithesis to a tree, as it has all available cycles between vertices of different generations. Works
exploiting this structure in the context of spectral properties of the Laplace operator include \cite{GHM, KLW, Woj}. The Anderson model on
antitrees has been studied in \cite{Sa, Sa1}.

It is shown in \cite{BrKe} that antitrees are family preserving.  A metric graph, $\Gamma$, is called an antitree if $G_\Gamma$ is an
antitree. In this subsection we would like to demonstrate the applicability of our method by applying it to spherically homogeneous metric antitrees. Thus, let $\Gamma$ be a spherically homogeneous metric antitree.

Antitrees are (generally) not locally balanced, as for every $n\in\mathbb{N}$ we have $|S_{n}|\cdot|S_{n+1}|$
edges of generation $n$ (which, if $|S_{n}|,|S_{n+1}|\geq2$, is greater than both $|S_n|$ and $|S_{n+1}|$). Thus, given an antitree $\Gamma$,
we consider
its unitarily equivalent graph $\widetilde{\Gamma}$ which is obtained
by adding vertices where necessary. For simplicity of notation, we assume here that
for every $n\geq1$, $|S_{n}|\geq2$ (note that $|S_0|=1$ since $S_0=\{\delta_o\}$--the root). The analysis of the general case is not
fundamentally different, but more cumbersome to describe.

Write $V(\widetilde{\Gamma})=V(\Gamma)\cup\widetilde{V}$,
and denote by $(\rho_{n})_{n\in\mathbb{N}}$ the set of indices
for which $S_{\rho_{n}}\subseteq V(\Gamma)$ (note that $\rho_{1}=0$, $\rho_{2}=1$,
$\rho_{3}=3$ etc.). Also, denote by $(\tilde{\rho}_{n})_{n\in\mathbb{N}}$ the
set of indices for which $S_{\tilde{\rho}_{n}}\subseteq\widetilde{V}$ (here, $\tilde{\rho}_{1}=2$,
$\tilde{\rho}_{2}=4$ etc.). Finally, let $\widetilde{G}$ be the discrete structure
of $\widetilde{\Gamma}$ and let $(H_{r})_{r=0}^{\infty}$ be the decomposition
of $\ell^{2}(\widetilde{G})$ described in Theorem \ref{thm_br}, such that $H_{0}=\overline{span\left\{
\delta_{o},\Delta_{d}(\delta_{o}),\Delta^{2}_{d}(\delta_{o}),\ldots\right\} }$
and for every $r\in\mathbb{N}$, $H_{r}=\overline{span\left\{ \phi_{0}^{r},\Delta_{d}(\phi_{0}^{r}),\Delta^{2}_{d}(\phi_{0}^{r}),\ldots\right\}
}$
for $\phi_{0}^{r}\in\ell^{2}(S_{n(r)})$ (i.e.\ $n(r)$ denotes the
sphere on which $\phi_{0}^{r}$ is supported). Recall that for every
$r\in\mathbb{N}$, $(\phi_{k}^{r})_{k=0}^{\infty}$ is the set obtained
by applying the Gram-Schmidt process on $\left(\Delta^{k}(\phi_{0}^{r})\right)_{k=0}^{\infty}$. We want to describe the decomposition of
$\Delta$ on $\widetilde{\Gamma}$. In order to do this, we first need to focus on the corresponding discrete decomposition.

\begin{claim}
	For every $n\geq3$ and for every $r\in\mathbb{N}$, $\rho_n \neq n(r)$.
\end{claim}

\begin{proof}
	By Theorem \ref{thm_br}, for every $\phi\in H_{r}$ we have that $supp(\phi)\cap S_{n(r)-1}=\emptyset$.
	Assume that there exists $n\geq3$ and $r\in\mathbb{N}$ such that
	$n(r)=\rho_{n}$. Note that, as the sphere $S_{\rho_{n}-1}$ consists of
	added vertices, every $v\in S_{\rho_{n}-1}$ is connected to exactly
	one $w\in S_{\rho_{n}}$. Thus, for $w\in S_{\rho_{n}}$ such that $\phi_{0}^{r}(w)\neq0$
	and a neighbor of $w$, $v\in S_{\rho_{n}-1}$, $\Delta_d(\phi_{0}^{r})(v)=-\phi_{0}^{r}(w)$
	which is a contradiction.
\end{proof}
\begin{claim}
	For every $3\leq n\in\mathbb{N}$ and for every $r\in\mathbb{N}$
	such that $0\leq n(r)\leq \rho_{n}-2$, $H_{r}\cap\ell^{2}(S_{\rho_{n}})=\emptyset$.
\end{claim}

\begin{proof}
	For every $0\leq k\leq \rho_{n}-2$ and $v\in S_{\rho_{n}}$, $G_{v}\cap S_{k}=S_{k}$.
	In addition, $\phi_{0}^{r}$ is orthogonal to the symmetric functions.
	Combining this with the fact that $\Delta_d^{\rho_{n}-n(r)}(\phi_{0}^{r})(v)=c\underset{w\in G_{v}\cap S_{n(r)}}{\sum}\phi_{0}^{r}(w)$
	for some $c\in\mathbb{N}$, we get the result.
\end{proof}
\begin{cor}
	For every $3\leq n\in\mathbb{N}$, $\ell^{2}(S_{\rho_{n}})=span\left\{
\phi_{\rho_{n}}^{0},\phi_{1}^{r_{1}},\phi_{1}^{r_{2}},\ldots,\phi_{1}^{r_{|S_{\rho_{n}|-1}}}\right\} $
	for $r_{1},\ldots,r_{|S_{\rho_{n}}|-1}$ such that $n(r_{i})=\rho_{n}-1$
	for every $1\leq i\leq|S_{\rho_{n}}|-1$. In addition, for every such
	$i$, $\Delta_d\left(\phi_{1}^{r_{i}}\right)|_{S_{\rho_{n}+1}}\neq{0}$.
\end{cor}

\begin{proof}
	The first part follows from the previous claims and from the fact
	that $\ell^{2}(S_{\rho_{n}})\subset\underset{r\in\mathbb{N}}{\oplus}H_{r}$.
	The second part follows from the fact that for $0\neq\phi\in\ell^{2}(S_{\rho_{n}})$,
	$\Delta_d(\phi)|_{S_{\rho_{n}+1}}\neq0$, as again, every vertex in $S_{\rho_{n}+1}$
	is connected to exactly one vertex in $S_{\rho_{n}}$.
\end{proof}
Corollary 5.3 implies that for every $n\geq3$, a subspace of $\ell^2\left(S_{\rho_{n}+1} \right)$,
of dimension $|S_{\rho_{n}}|-1$, is spanned by $\left\{ \phi_{1}^{r_{1}},\ldots,\phi_{1}^{r_{|S_{\rho_{n}}|-1}}\right\} $.
For every $r\in\mathbb{N}$ such that $n(r)<\rho_{n}-1$, $H_{r}\cap \ell^2(S_{\rho_{n}+1})=\emptyset$.
This follows, as before, from the fact that for every $v\in S_{\rho_n+1}$,
$G_{v}\cap S_{n(r)}=S_{n(r)}$. Another one-dimensional space is spanned
by $\phi_{_{\rho_{n}+1}}^{0}$. Recalling that $|S_{\rho_{n}+1}|=|S_{\rho_{n}}|\cdot|S_{\rho_{n+1}}|$,
this means that $\left|\left\{ r\in\mathbb{N}\,|\,n(r)=\rho_{n}+1\right\}
\right|=|S_{\rho_{n+1}}|\cdot|S_{\rho_{n}}|-|S_{\rho_{n}}|=|S_{\rho_{n}}|\cdot\left(|S_{\rho_{n+1}}|-1\right)$
(as these span the rest of $\ell^{2}(S_{\rho_{n}+1})$). Noting
that for every $n\geq2$, $\rho_{n}+1=\tilde{\rho}_{n-1}$, we conclude the following.
\begin{prop}
	Denote $k_{n}=|S_{\rho_{n}}|\cdot\left(|S_{\rho_{n+1}}|-1\right)$, and
	for $n\geq2$, $B_{n}=\left\{ r\in\mathbb{N}\,|\,n(r)=\tilde{\rho}_{n-1}\right\} $.
	Then for every $n\geq2$, $\left|B_{n}\right|=k_{n}$. Furthermore,
	define $K_{n}\coloneqq\left\{ r\in B_{n}\,|\,dim(H_{r})=3\right\} $.Then
	$|K_{n}|=|S_{\rho_{n+1}}|-1$, and for every $r\in B_{n}\setminus K_{n}$,
	$dim(H_{r})=1$.
\end{prop}

Denote by $Q$ the set $\left\{ r\in\mathbb{N}\,|\,\exists n\in\mathbb{N}\text{ s.t. \ensuremath{n(r)=\tilde{\rho}_{n}}}\right\} \cup\{0\}$,
and define $W=\underset{r\in Q}{\oplus}H_{r}$. From all of the above,
we conclude the following:
\begin{itemize}
	\item For $n\geq3$, $\ell^{2}(S_{n})\subseteq W$.
	\item For $n=2$, $dim\left(\ell^{2}(S_{n})\cap W\right)=|S_{1}|\cdot|S_{3}|-|S_{1}|+1$.
	\item For $n=0,1$, the intersection of $W$ with $\ell^{2}(S_{n})$ consists
	of the symmetric functions on $S_{n}$ (for $S_{0}$ this is actually
	the whole space, as $|S_{0}|=1$).
\item Finally, for $r\notin Q$ $H_{r}$ is orthogonal to $W$. Thus, the
intersection of $H_{r}$ with $\ell^{2}(S_{n})$ is trivial for every
$n\geq3$ and for $n=0$. This implies that for every $r\notin Q$,
$n(r)=1$ and there will be exactly $|S_{1}|-1$
such $r$'s.
\end{itemize}

By considering the spaces $\left(F_{r}\right)_{r=0}^{\infty}$
given by Theorem \ref{thm_main} we can now state the analog of \cite[Theorem 1.3]{BrKe}

\begin{thm} \label{thm_antitrees}
Let $\Gamma$ be an infinite spherically homogeneous metric antitree. Then the spectrum of $\Delta$ on $\Gamma$ is given by the spectrum of a
weighted Laplacian on the whole line and the closure of the union of spectra of countably many weighted Laplacians on compact segments.
\end{thm}

\begin{proof}
Recall that for $n\in\mathbb{N}$, $t_{n}$
denotes the metric distance of $v$ from $o$ for some $v\in S_{n}$.
In addition, recall that $a_{r}$ and $b_{r}$ denote the part of
the real line on which functions in $F_{r}$ are supported.

Now, note that for every $x\in\widetilde{\Gamma}$, $G_{x}\cap S_{0}=\left\{ o\right\}$.	
This means that $h^{\delta_{o}}(x)\neq0$. Thus, $b_{0}=\infty$ and $\Delta$ restricted to $F_0$ is unitarily equivalent to a weighted
Laplacian on the positive real line, with matching conditions as described in the previous section on the points
$\left(t_{n}\right)_{n\in\mathbb{N}}$.

For every $r_{0}\in\left\{ r\in\mathbb{N}\,|\,n(r)=1\right\} $, $F_{r_{0}}$ is orthogonal to the symmetric functions. Thus, as every $v\in
S_{3}$ is connected to every vertex in $S_{1}$ (through a vertex in $S_2$), $\Delta$ restricted to $F_{r_0}$ is unitarily equivalent to a
weighted Laplacian on the compact segment $\left[0,t_{3}\right]$.

Finally, let $n\in\mathbb{N}$ and let $r_{0}\in\left\{ r\in\mathbb{N}\,|\,n(r)=\tilde{\rho}_{n}\right\}$ and consider two cases.
	\begin{itemize}
		\item $dim(H_{r_{0}})=1$: In this case, $b_{r}=t_{\tilde{\rho}_{n}+1}$, and we get
		a copy of a weighted Laplacian on the compact segment $\left[t_{\tilde{\rho}_{n}-1},t_{\tilde{\rho}_{n}+1}\right]$.
		\item $dim(H_{r_{0}})=3$: In this case, $b_{r}=t_{\tilde{\rho}_{n}+3}$, and we get
		a copy of a weighted Laplacian on the compact segment $\left[t_{\tilde{\rho}_{n}-1},t_{\tilde{\rho}_{n}+3}\right]$.
	\end{itemize}
\end{proof}
\begin{rem}\label{rem_antitrees}
In an antitree, there is an explicit formula for $g_{0}^{j}$ which is given by $g_{0}^{j}=|S_j|\cdot|S_{j+1}|$. Thus, with Remark \ref{rem_dj}
in mind, $d_{j}^{0}=\sqrt{\frac{|S_{j+1}|}{|S_{j-1}|}}$.
\end{rem}
As a demonstration of the usefulness of an explicit decomposition to one dimensional operators, we mention Remling's Theorem \cite{remling},
whose main message is the fact that absolutely continuous spectrum for one dimensional operators is extremely restrictive. In the context of
antitrees, following the discussion above and applying \cite[Theorem 6]{BF} we immediately obtain that
\begin{thm} \label{thm_antitrees}
	Let $\Gamma$ be a spherically homogeneous metric antitree and assume that
	\begin{equation} \nonumber
	\inf_n (t_{n+1}-t_n)>0
	\end{equation}
	and
	\begin{equation} \label{eq:antitree}
	\underset{n\rightarrow \infty}{\liminf}\  \frac{|S_{n+1}|}{|S_{n-1}|}>1.
	\end{equation}
	Then if $\limsup_{n \rightarrow \infty}(t_{n+1}-t_n)=\infty$ then the absolutely continuous spectrum of $\Delta$ on $\Gamma$ is empty.
\end{thm}

\begin{rem}
Theorem \ref{thm_antitrees} coincides with (part of) Theorem 8.1 in the recent preprint \cite{KN}, which has an extensive analysis of spherically homogeneous metric antitrees, also using a decomposition inspired by \cite{BrKe}. Using a modification of the methods of \cite{BF}, it is likely one could replace \eqref{eq:antitree} with $\underset{n\rightarrow \infty}{\liminf}\  \left| \frac{|S_{n+1}|}{|S_{n-1}|}-1 \right|>0$ and still obtain absence of absolutely continuous spectrum.
\end{rem}

\section{Appendix -Proof of Proposition \ref{u_eq_loc_bal}}

\begin{proof}[Proof of Proposition \ref{u_eq_loc_bal}]
We say that $n\in\mathbb{N}$ is ``bad'' if the number of edges
of generation $n$ is greater than the number of vertices of generation
$n$ and of the number of vertices of generation $n+1$. For such
$n$, and $e=(i(e),t(e))$ such that $gen(e)=n$, we add a vertex
$w$, the edges $e_{1}=(i(e),w)$ and $e_{2}=(w,t(e))$ such that $l(e_{i})=\frac{l(e)}{2}$,
and remove $e$ from $E(\Gamma)$. In simple words, we divide the
edge into two edges with equal length. Note that in the resulting
graph, the numbers $n,\,n+1$ are not bad, because for every new edge
$e$, there is a unique vertex $w_{e}$ which was added with $e$.
In addition, none of the other generations became bad, because we
did not remove any vertices from $V(\Gamma)$. Thus, if we follow
this procedure for every bad $n$, the resulting graph will be locally
balanced. Denote that graph by $\overline{\Gamma}$, and let $\overline{\Delta}=\Delta_{\overline{\Gamma}}$.
Note that as measure spaces, there is no difference between $\Gamma$
and $\overline{\Gamma}$. The Kirchhoff boundary conditions assure
us that the transformation $T:\,D(\Delta)\rightarrow D(\overline{\Delta})$
defined by $T(\varphi)=\varphi$ is unitary, and maps $D(\Delta)$
bijectively onto $D(\overline{\Delta}).$

It is left to prove that $\overline{\Gamma}$ is family preserving.
Denote by $\overline{S_{n}}$ the n'th sphere in $\overline{\Gamma}$.
Let $v,u\in\overline{S_{n}}$ such that $v$ and $u$ are forward
neighbors. Consider the following cases:

$(i)$ $\overline{S_{n}}$ is not a part of $\Gamma$, meaning we
added that sphere during the above process. In this case, $v$ and
$u$ were added in the middle of the edges $(w_{1},z)$, $(w_{2},z)$
which means that $w_{1}$ and $w_{2}$ are forward neighbors in $\Gamma$.
That implies that there exists a rooted graph automorphism $\tau:\,V(\Gamma)\rightarrow V(\Gamma)$
such that $\tau(w_{1})=\tau(w_{2})$, and $\forall k\geq0\,\tau|_{S_{n+k}}=Id$.
We define $\overline{\tau}:\,V(\overline{\Gamma})\rightarrow V(\overline{\Gamma})$
by $\overline{\tau}|_{V(\Gamma)}=\tau$, and if $x$ was added in
the middle of the edge $(v_{1},v_{2})$, $\overline{\tau}(x)$ will
be the vertex that was added in the middle of $(\tau(v_{1}),\tau(v_{2}))$
(such vertex exists because of spherical symmetry). Now, we have that
$\overline{\tau}(v)=u$, and $\forall k\geq0\,\overline{\tau}|_{\overline{S_{n+k}}}=Id$
, as required.

$(ii)$ $\overline{S_{n}}$ is a part of $\Gamma$. In that case,
$\overline{S_{n+1}}$ is also a part of $\Gamma$, because an added
vertex cannot connect vertices from the preceding sphere, so we can
define $\overline{\tau}$ exactly as in case $(i)$, and get the desired
automorphism.

The proof for backward neighbors is exactly the same.
\end{proof}


\end{document}